\documentclass{article}
\usepackage[utf8]{inputenc}
\usepackage[a4paper,top=2.5cm,bottom=3cm,left=2.5cm,right=2.5cm,marginparwidth=2cm]{geometry}
\usepackage[hidelinks]{hyperref}
\usepackage{amsmath, amssymb, amsthm, bbm, graphicx, caption, tikz, subcaption, siunitx, enumerate}
\newcommand{\F}{\mathcal{F}}
\newcommand{\tr}{\mathrm{tr}}
\renewcommand{\S}{\mathcal{S}}
\newcommand{\sgn}{\mathrm{sgn}}

\newtheorem{thm}{Theorem}

\newtheorem{lemma}[thm]{Lemma}

\newtheorem{cor}[thm]{Corollary}
\theoremstyle{definition}

\numberwithin{equation}{section}
\numberwithin{thm}{section}

\title{Super band gaps and periodic approximants of generalised Fibonacci tilings}
\author{Bryn Davies\thanks{Department of Mathematics, Imperial College London, London SW7~2AZ, UK.}
  \and Lorenzo Morini\thanks{School of Engineering, Cardiff University, Cardiff CF24~3AA, UK.}}
\date{}

\usepackage{todonotes}

\begin{document}
\maketitle

\begin{abstract}
    We present mathematical theory for understanding the transmission spectra of heterogeneous materials formed by generalised Fibonacci tilings. Our results, firstly, characterise super band gaps, which are spectral gaps that exist for any periodic approximant of the quasicrystalline material. This theory, secondly, establishes the veracity of these periodic approximants, in the sense that they faithfully reproduce the main spectral gaps. We characterise super band gaps in terms of a growth condition on the traces of the associated transfer matrices. Our theory includes a large family of generalised Fibonacci tilings, including both precious mean and metal mean patterns. We demonstrate our fundamental results through the analysis of three different one-dimensional wave phenomena: compressional waves in a discrete mass-spring system, axial waves in structured rods and flexural waves in multi-supported beams. In all three cases, the theory is shown to give accurate predictions of the super band gaps, with negligible computational cost and with significantly greater precision than previous estimates.
\end{abstract}

\section{Introduction}

Heterogeneous materials have the ability to manipulate and guide waves in carefully controlled ways. The discovery of exotic phenomena, such as negative refraction and cloaking effects, led to the name \emph{metamaterials} being coined \cite{kadic20193d}. While many such materials are based on periodic structures, quasiperiodic materials have fascinating wave scattering and transmission properties and have the potential to greatly enlarge the metamaterial design space. However, the lack of concise mathematical methods able to describe the transmission spectra of quasiperiodic materials efficiently and with minimal computational cost is a significant barrier to widespread usage. In this work, we help to overcome this barrier by developing a concise approach for characterising the spectral gaps in quasicrystalline generated materials.

Characterising the spectra of quasiperiodic differential operators is a longstanding and fascinating problem. In particular, one-dimensional Schr\"odinger operators with quasiperiodic potentials have been widely studied. Typical results concern the Cantor-type properties of the spectrum \cite{avila2009ten, eliasson1992floquet, gumbs1988dynamical, kohmoto1984cantor} and the extent to which its spectrum can be decomposed into pure-point, singularly continuous and absolutely continuous eigenvalues \cite{jitomirskaya1999metal, surace1990schrodinger}. In this work, the aim is to quantify specific spectral features, rather than characterise overall properties of the spectrum. A promising avenue in this direction, which we will not make use of in this work, is to exploit the fact that quasicrystals can be obtained through incommensurate projections of periodic patterns in higher dimensional spaces. This approach has been used to model wave propagation in one-dimensional quasicrystals \cite{amenoagbadji2023wave} and make predictions on the locations of spectral gaps \cite{rodriguez2008computation}. In the latter case, this approach has suffered from the occurrence of spurious modes and a precise convergence theory has yet to be established. In this work, we will bypass these issues by taking a different approach that is specifically developed for generalised Fibonacci quasicrystals.

Generalised Fibonacci tilings are a subclass of the family of one-dimensional quasiperiodic patterns that can be generated by substitution rules. These patterns were classified by \cite{kolar1993new} and are formed by tiling two distinct elements, labelled $A$ and $B$, according to some substitution rule
\begin{equation} \label{eq:tiling_general}
    A\to \mathcal{M}_{ml}(A,B), \quad B\to \mathcal{M}_{{m'}{l'}}(A,B),
\end{equation}
where $\mathcal{M}_{ml}(A,B)$ is some pattern that contains the $A$ elements $m$ times and  the $B$ elements $l$ times. The most widely studied example of such a tiling is the golden mean Fibonacci tiling, which is given by \eqref{eq:tiling_general} with $m=l=m'=1$ and $l'=0$. The first few terms of this sequence are shown in Figure~\ref{fig:FibonacciTiling}. Generalised Fibonacci tilings have been studied extensively in the literature for various elastic, mechanical and Hamiltonian systems \cite{dal2007spectral, gei2010, gei2020phononic, hiramoto1989new, kraus2012topological, morini2018waves, moustaj2022spectral}. Complex patterns of stop and pass bands have been observed, whose features include large stop bands across multiple frequency scales and self similar properties.

\begin{figure}
    \centering
    \includegraphics[width=\linewidth]{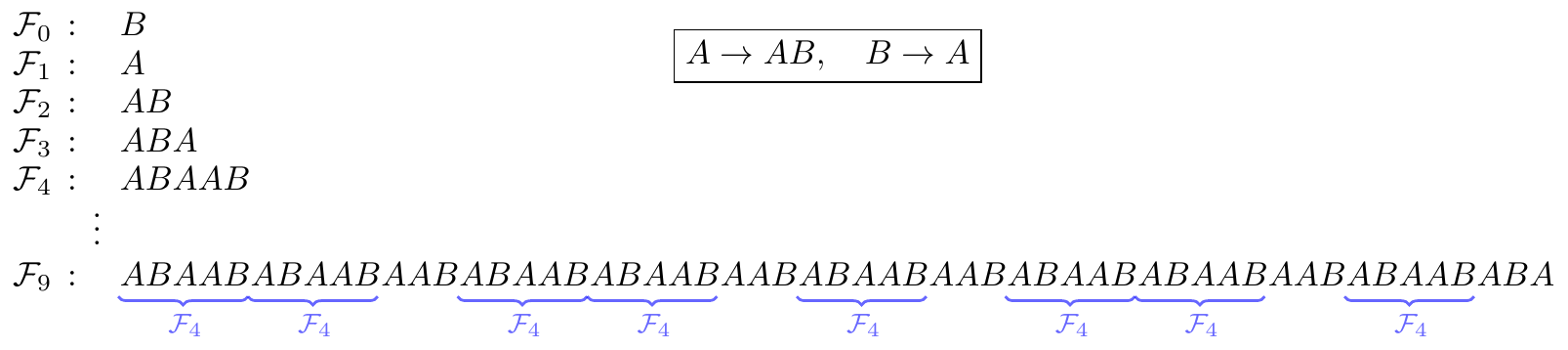}
    \caption{The golden mean Fibonacci tiling (where $m=1$ and $l=1$). The self similarity of the structures is clear from later terms in the sequence. As an example, $\F_9$ is contains many repetitions of $\F_4$.}
    \label{fig:FibonacciTiling}
\end{figure}

Given the challenges of characterising the spectra of quasicrystals, a common strategy is to consider periodic approximants of the material, sometimes known as \emph{supercells}. This approach is commonplace in the physical literature (for example, in \cite{chan1998photonic, florescu2009complete, hamilton2021effective}) and has the significant advantage that the spectra of the periodic approximants can be computed efficiently using Floquet-Bloch analysis. This method characterises the spectrum as a countable collection of spectral bands with \emph{band gaps} between each band. In the setting of tilings where the quasicrystalline pattern is generated using a substitution rule, such an approach is particularly promising. A natural question to ask is how the band gaps evolve as the unit cell is grown according to the given tiling rule. An example is shown in Figure~\ref{fig:intro_SBG}, where we plot the band diagrams for a system of axial waves in structured rods (which will be examined in detail in Section~\ref{sec:rods}) with the unit cell is designed to follow the golden mean Fibonacci tiling. We can see that while the spectrum of the Fibonacci tilings $\F_n$ becomes increasingly complex as $n$ grows, there are some clear features that emerge. As $n$ increases, the pattern of pass bands and band gaps becomes increasingly fragmented, reminiscent of the Cantor-type behaviour predicted by the literature for other quasiperiodic operators \cite{avila2009ten, eliasson1992floquet, gumbs1988dynamical, kohmoto1984cantor}. In spite of this complexity, several large band gaps seem to appear for relatively small $n$ (\emph{e.g.} for $\F_4$) and persist as $n$ grows. These features were noticed by \cite{morini2018waves} who coined the phrase \emph{super band gaps} to describe these features.

\begin{figure}
    \centering
    \begin{subfigure}{0.48\textwidth}
        \includegraphics[width=\linewidth,trim=1.1cm 0 1.5cm 0,clip]{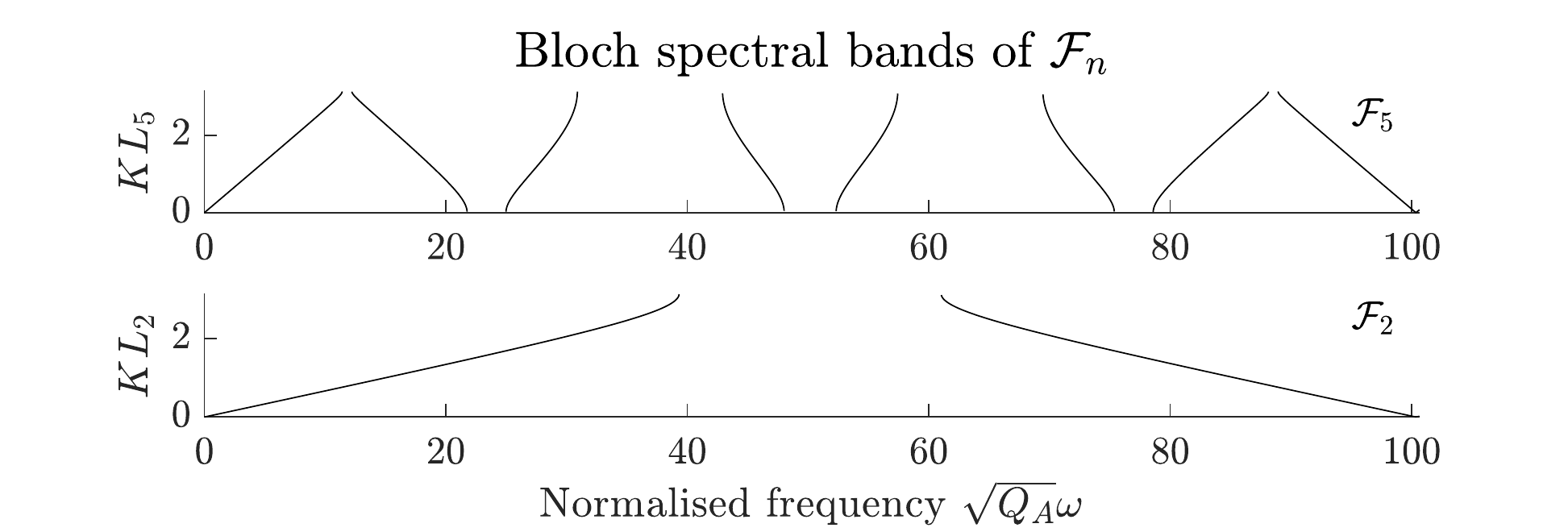}
    \end{subfigure}
    \hfill
    \begin{subfigure}{0.48\textwidth}
        \includegraphics[width=\linewidth,trim=1.1cm 0 1.5cm 0,clip]{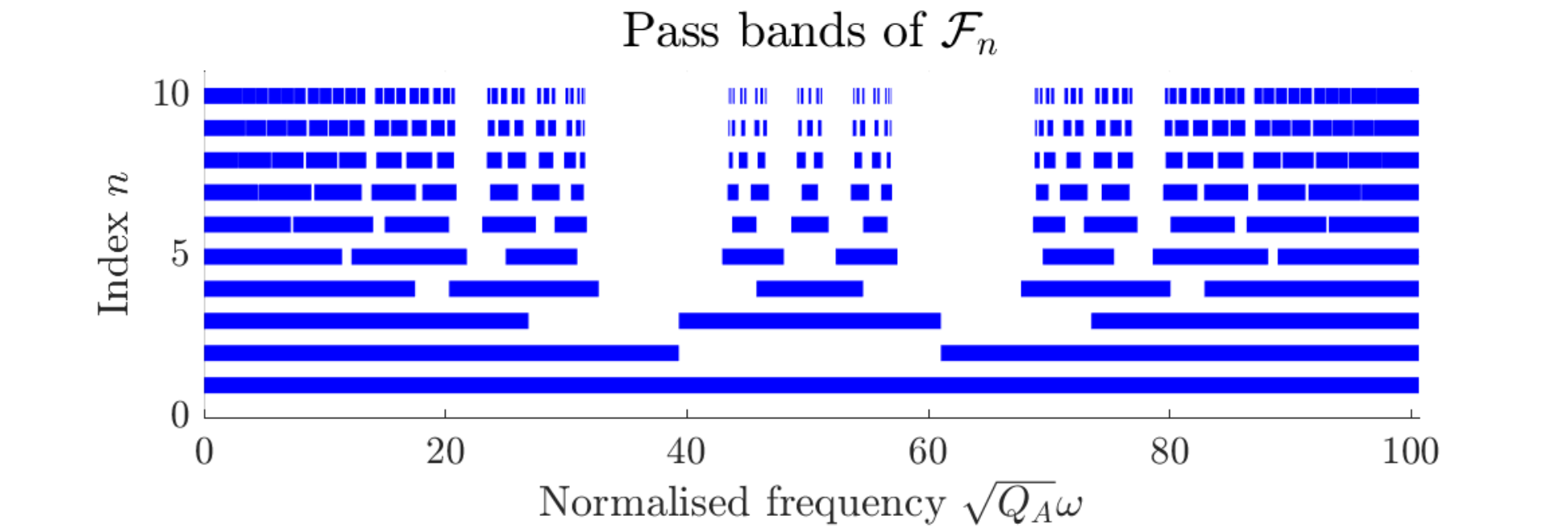}
    \end{subfigure}
    
    \caption{The transmission spectra of periodic structures with unit cells given by generalised Fibonacci tilings can be computed using Floquet-Bloch. Left: The Bloch band diagrams for periodic unit cells given by the golden mean Fibonacci tilings $\F_2$ and $\F_5$. Right: The pass bands for successive golden mean Fibonacci tilings, from which the emergence of super band gaps is clear.}
    \label{fig:intro_SBG}
\end{figure}

One explanation for the emergence of super band gaps in Fibonacci tilings is their structural self similarity. For example, Figure~\ref{fig:FibonacciTiling} depicts the first few golden mean Fibonacci tilings and it is clear that $\F_9$ contains $\F_4$ repeated many times, with a few other interspersed pieces. Thus, it is to be expected that a periodic material with $\F_9$ as its unit cell should share some of the main features of its transmission spectrum with the $\F_4$ periodic material. \cite{morini2018waves} developed a simple but successful approximation strategy for predicting the approximate locations of super band gaps in generalised Fibonacci tilings. However, a theoretical justification of this behaviour remains an open question. The aim of this work is to develop mathematical theory to characterise the existence of these super band gaps.

Understanding super band gaps is not only useful for characterising the main features of the transmission spectra of quasicrystalline materials, but also provides justification for the use of periodic approximants (supercells). We will demonstrate in Section~\ref{sec:periodic} that the transmission coefficient of a finite-sized piece of a Fibonacci quasicrystal can be approximated by the transmission spectrum of a periodic approximant. Our results show that even a periodic approximant with a small unit cell can accurately predict the main spectral gaps of the finite one-dimensional quasicrystal. This is predicted by our theory for super band gaps, which demonstrates the existence of frequency ranges which will always be in spectral gaps, for any generalised Fibonacci tiling beyond a given term in the sequence.

The methods developed in this study will apply to one-dimensional wave systems with two degrees of freedom, which can be described by a $2\times2$ transfer matrix. Three examples of applicable systems are shown in Figure~\ref{fig:systems}. The first is a discrete system of masses and springs, where we vary the spring constants and the masses to give the appropriate $A$ and $B$ pattern. The second system concerns axial waves in structured rods, which are governed by a Helmholtz equation. Here, we modulate the dimensions and also the material parameters (Young's modulus and mass density). Finally, we will consider a continuous flexural beam that is supported at varying intervals. We will examine these three systems in detail in Section~\ref{sec:examples} and present numerical results demonstrating that our theory for super band gaps can be used to reveal spectral features accurately and with minimal computational cost.

\begin{figure}
    \centering
    \includegraphics[width=0.95\linewidth]{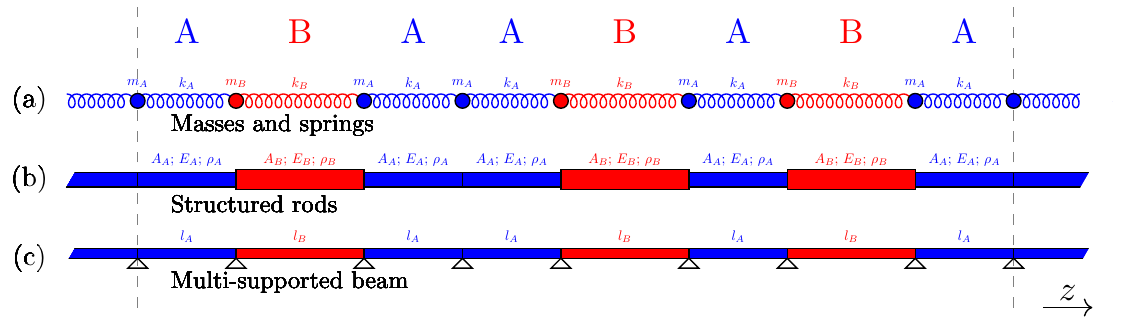}
    \caption{We will present numerical results for three different one-dimensional wave systems in this work. Here, unit cells corresponding to the golden mean Fibonacci tiling $\F_5=ABAABABA$ is shown. (a) A discrete system of masses coupled with springs, where we modulate both the masses $m_X$ and the spring constants $k_X$. (b) Axial waves in structured rods, where the cross sections $A_X$, the Young's modulus $E_X$ and the mass density $\rho_X$ can be modulated. (c) Flexural waves in multi-supported beams, where the distances $l_X$ between the supports are varied.}
    \label{fig:systems}
\end{figure}

\section{Generalised Fibonacci tilings} \label{gentheory}

Generalised Fibonacci structures are defined according to the substitution rule

\begin{equation} \label{eq:tiling}
    A\to A^m B^l, \quad B\to A,
\end{equation}
where $m$ and $l$ are positive integers. Typically, the sequence is initiated with $\F_0=B$, which yields that $\F_1=A$, $\F_2=A^mB^l$, $\F_3=(A^mB^l)^mA^l$ and so on (see Figure~\ref{fig:FibonacciTiling}). The total number of elements in $\F_n$ is given by the $n$\textsuperscript{th} generalised Fibonacci number $F_n$, which are defined according to the recurrence relation
\begin{equation} \label{eq:Fibonaccinumbers}
    F_n = mF_{n-1} + lF_{n-2}.
\end{equation}
The limit of the ratio $F_{n+1}/F_n$ as $n\to\infty$ is given by
\begin{equation} \label{def:sigma}
    \sigma(m,l) := \lim_{n\to\infty} \frac{F_{n+1}}{F_n}=\frac{m+\sqrt{m^2+4l}}{2}
\end{equation}
and the tilings inherit their names from this limiting ratio. For example, since $\sigma(1,1)=(1+\sqrt{5})/2\approx1.618\dots$, this case is often known as the golden mean Fibonacci tiling. Similarly, $\sigma(2,1)=1+\sqrt{2}\approx2.414\dots$ is the silver mean and $\sigma(3,1)=(3+\sqrt{13})/2\approx 3.303\dots$ is the bronze mean. Likewise, $\sigma(1,2)$ and $\sigma(1,3)$ have assumed the names copper mean and nickel mean, respectively.

We will study wave propagation in systems that have two degrees of freedom, in the sense that their behaviour can be described fully by a two-element state vector $\mathbf{u}_j\in\mathbb{R}^2$, where $j$ is an index denoting the spatial position. We suppose that wave propagation in these systems can be described by a unimodular transfer matrix $T(\omega)$ with real-valued entries. That is, for any indices $i$ and $j$ and any frequency $\omega$ there is some matrix $T(\omega)\in\mathbb{R}^{2\times2}$ such that $\det(T)=1$ and $\mathbf{u}_i=T(\omega) \mathbf{u}_j$. We will explore three different examples of such systems in Section~\ref{sec:examples}.

We let $T_n$ be the transfer matrix associated with the structure $\F_n$. The substitution rule \eqref{eq:tiling} means that this sequence of transfer matrices satisfies 
\begin{equation}
    T_{n+1} = T_{n-1}^l \, T_{n}^m.
\end{equation}
We are interested in studying structures formed by repeating $\F_n$ periodically. We can relate the state vector at either ends of the unit cell $\F_n$ by
\begin{equation} \label{uFngen}
\mathbf{u}_{F_{n}}=T_n(\omega)\mathbf{u}_{0}.
\end{equation}
Then, to understand the transmission properties of the periodic material, we can apply the Floquet-Bloch theorem. If $L_n$ is the length of the unit cell $\F_n$, then we substitute $\mathbf{u}_{F_{n}}=\mathbf{u}_0e^{iKL_n}$ into equation (\ref{uFngen}), giving that $\det(T_n(\omega)-e^{iKL_n}I)=0$. Using the fact that $\det(T_n)=1$, this reduces to the simple dispersion relation
\begin{equation} \label{dispersion}
\cos({KL_n})=\frac{1}{2}\tr (T_n({\omega})).
\end{equation}
This has a real solution for $K$ if and only if $|\tr (T_n({\omega}))|\leq 2$. If $\omega$ is such that $K$ is complex, then we do not have Floquet-Bloch modes so $\omega$ lies in a band gap of the periodic material. Examples of the dispersion diagrams obtained by solving \eqref{dispersion} for the Fibonacci tilings $\F_2$ and $\F_5$ are shown in Figure~\ref{fig:intro_SBG}.

Characterising the band gaps of the material reduces to finding $\omega$ such that $|\tr (T_n({\omega}))|> 2$. Given the importance of the transfer matrix trace, we define the quantity
\begin{equation}
    x_n(\omega) = \tr(T_n(\omega)).
\end{equation}
Understanding how the sequence $\{x_n(\omega):n=1,2,\dots\}$ evolves for different materials and at different frequencies $\omega$ will be the main theoretical challenge tackled in this work. In particular, we will define a super band gap to be the set $\S_N$ of all $\omega\in\mathbb{R}$ which are in band gaps of $\F_n$ for all $n\geq N$. That is
\begin{equation}
    \S_N:=\left\{ \omega\in\mathbb{R} : |x_n(\omega)|>2 \text{ for all } n\geq N \right\}.
\end{equation}

In this work, we will characterise super band gaps in Fibonacci tiling by deriving ``growth conditions'' that guarantee a frequency being in a super band gap. These results say that if $\omega$ is such that there exists some $N\in\mathbb{N}$ for which $|x_{N}(\omega)|>2$ and the following terms $|x_{N+1}(\omega)|$ and $|x_{N+2}(\omega)|$ grow sufficiently quickly (in a sense that will depend on the choice of tiling parameters $l$ and $m$), then $\omega$ is guaranteed to be in the super band gap $\S_N$. This analysis will rest upon the helpful observation that the traces corresponding to generalised Fibonacci tilings satisfy recursive relations \cite{kolar1989generalized, kolar1990trace}. To state these recursion relations, we must first introduce the quantity 
\begin{equation} \label{defn:tn}
    t_n(\omega) := \tr( T_{n-2}(\omega) T_{n-1}(\omega) ).
\end{equation}
We will also need the sequence of polynomials $d_k(x)$, defined recursively by
\begin{equation} \label{defn:dk}
    d_0(x)=0, \quad
    d_1(x)=1 \quad \text{and}\quad 
    d_{k}(x) = x d_{k-1}(x)-d_{k-2}(x) \,\text{ for }\, k\geq2.
\end{equation}
We have that $d_2(x)=x$, $d_3(x)=x^2-1$, $d_4(x)=x^3-2x$, $d_5(x)=x^4-3x^2+1$ and so on. These polynomials are rescaled Chebyshev polynomials of the second kind. Understanding the properties of these polynomials (in Section~\ref{sec:polynomials}) will be one of the key insights that will allow us to prove spectral properties of generalised Fibonacci tilings for large values of $m$ or $l$. Finally we have the following recursion relation describing the evolution of $x_n$ and $t_n$, which was shown by \cite{kolar1990onedim}
\begin{equation} \label{rec:general}
    \begin{cases}
        x_{n+1}=d_m(x_n)[d_l(x_{n-1})t_{n+1}-d_{l-1}(x_{n-1})x_n]-d_{m-1}(x_n)[d_{l+1}(x_{n-1})-d_{l-1}(x_{n-1})],\\
        t_{n+1}=d_{m+1}(x_{n-1})[d_l(x_{n-2})t_n-d_{l-1}(x_{n-2})x_{n-1}] -d_m(x_{n-1})[d_{l+1}(x_{n-2})-d_{l-1}(x_{n-2})].
    \end{cases}
\end{equation}

The name ``super band gap'' was introduced by \cite{morini2018waves}, who observed their existence in generalised Fibonacci structures (corresponding to the golden and silver means). They succeed in predicting the approximate locations of these super band gaps using the function $H_n:\mathbb{R}\to[0,\infty)$ defined by
\begin{equation} \label{eq:Hdefn}
    H_n(\omega)=|\tr(T_n(\omega))\tr(T_{n+1}(\omega))|.
\end{equation}
They observed numerically that if $\omega\in\mathbb{R}$ is such that $H_2(\omega)\gg2$, then it is likely to be in a super band gap. Other approximate approaches for predicting the locations of super band gaps also exist, such as considering an ``effective lattice'' that is the superposition of two periodic lattices, with periods differing by a ratio equal to the golden mean \cite{hamilton2021effective}. This work builds on these previous results by developing the first rigorous justification for the occurrence of super band gaps in materials generated by generalised Fibonacci tilings.

\section{Theory of super band gaps}  \label{gapstheory}

In this section, we will develop the main theory characterising super band gaps in materials generated by generalised Fibonacci tilings. These results will take the form of growth conditions, which will need to be modified to suit different values of $m$ and $l$. We will apply this theory to specific physical examples in Section~\ref{sec:examples} and use it to demonstrate the accuracy of periodic approximants in Section~\ref{sec:periodic}.

\subsection{Golden mean Fibonacci}

This is the classical Fibonacci tiling, where $m=1$ and $l=1$ in \eqref{eq:tiling}. It is referred to as the \emph{golden mean} Fibonacci tiling because the limiting ratio is $\sigma(1,1)=(1+\sqrt{5})/2\approx 1.618$, the famous golden mean that appears in nature. In the golden mean Fibonacci tiling, the recursion relation \eqref{rec:general} can be simplified to a much simpler form, given by
\begin{equation} \label{rec:golden}
    x_{n+1}=x_n x_{n-1} - x_{n-2}, \quad n\geq2.
\end{equation}
This was discovered by \cite{kohmoto1983localization} and has been the basis of many subsequent studies of Fibonacci materials. 

The main result we will use to characterise super band gaps is the following theorem. This shows that if a frequency is such that the sequence of traces is outside of $[-2,2]$ and has three subsequent terms that are growing, then that frequency is in a super band gap of the golden mean Fibonacci tiling. This result is a modification of the Lemma~3.3 in \cite{davies2022symmetry}, where it was proved for the special case where successive terms are double the previous term (giving exponential growth of the sequence). Here, we have improved the tightness of the bound and shown that any growth rate bigger than 1 is sufficient for a super band gap to exist.

\begin{thm} \label{thm:gold}
    Let $\omega\in\mathbb{R}$ and consider $x_n(\omega)$ satisfying the golden mean recursion relation \eqref{rec:golden}. Suppose that there exists some $N\in\mathbb{N}$ such that
    \begin{equation*}
        |x_{N}|>2, \quad
        |x_{N+1}|\geq|x_{N}|\quad\text{and}\quad
        |x_{N+2}|\geq|x_{N+1}|.
    \end{equation*}
    Then $|x_{n+1}|\geq|x_{n}|$ for all $n>N$. Consequently, $|x_n|>2$ for all $n\geq N$, meaning that $\omega$ is in the super band gap $\S_N$.
\end{thm}
\begin{proof}
We will show that $|x_{N+3}|>|x_{N+2}|$, from which the result will follow by induction. We have that
\begin{equation}
    |x_{N+3}|\geq |x_{N+2}||x_{N+1}|-|x_N| \geq |x_{N+2}||x_{N}|-|x_N| = |x_{N+2}|(|x_{N}|-1)+(|x_{N+2}|-|x_N|).
\end{equation}
By hypothesis, we have that $|x_{N}|-1>1$ and $|x_{N+2}|\geq|x_N|$, so it holds that $|x_{N+3}|\geq|x_{N+2}|$.
\end{proof}

\subsection{Silver mean Fibonacci}

The case  where $m=2$ and $l=1$ in \eqref{eq:tiling} is known as the \emph{silver mean} Fibonacci, again inheriting its name from the limit $\sigma(2,1)=1+\sqrt{2}\approx 2.414$. After some rearrangement, the corresponding recursion rule is given by
\begin{equation} \label{rec:silver}
    \begin{cases}
        x_{n+1} = x_n t_{n+1} - x_{n-1}, \\
        t_{n+1} = x_n x_{n-1} - t_n,
    \end{cases}
\end{equation}
for $n\geq 2$. While this is more complicated than in the case of the golden mean, we nevertheless have an analogous result to characterise super band gaps.

\begin{thm} \label{thm:silver}
    Let $\omega\in\mathbb{R}$ and consider $x_n(\omega)$ satisfying the silver mean recursion relation \eqref{rec:silver}. Suppose that there exists some $N\in\mathbb{N}$ such that
    \begin{equation*}
        |x_{N}|>2, \quad
        |x_{N+1}|\geq|x_{N}|\quad\text{and}\quad
        |x_{N+2}|\geq|x_{N+1}|.
    \end{equation*}
    Then $|x_{n+1}|\geq|x_{n}|$ for all $n>N$.
    Consequently, $|x_n|>2$ for all $n\geq N$, meaning that $\omega$ is in the super band gap $\S_N$.
\end{thm}
\begin{proof}
As for the golden mean Fibonacci tiling, the strategy will be to proceed by induction. We begin with the second equation of the recursion relation \eqref{rec:silver}, with a view to deriving a lower bound on $|t_{N+3}|$. Observe, first, that thanks to elementary properties of unimodular matrices
\begin{align} \label{eq:tn1}
    t_n = \tr(T_{n-2}T_{n-1}) 
    \leq \frac{1}{2} (\tr(T_{n-2}^2)+\tr(T_{n-1}^2))
    = \frac{1}{2}(x_{n-2}^2+x_{n-1}^2)-2,
\end{align}
for any $n$. In particular, since $|x_{N}|>2$ and $|x_{N+1}|>2$, the right hand side of \eqref{eq:tn1} is positive when $n=N+2$, so we have that 
\begin{equation}
    |t_{N+2}|\leq \frac{1}{2}(x_{N}^2+x_{N+1}^2)-2
    \leq x_{N+1}^2-2.
\end{equation}
Then, the second equation of \eqref{rec:silver} gives
\begin{equation}
    |t_{N+3}|\geq |x_{N+2}||x_{N+1}|-|t_{N+2}|
    \geq x_{N+1}^2-x_{N+1}^2+2=2.
\end{equation}
Finally, turning to the first equation of \eqref{rec:silver}, we see that 
\begin{equation}
    |x_{N+3}|\geq|x_{N+2}||t_{N+3}|-|x_{N+1}|
    \geq2|x_{N+2}|-|x_{N+1}|
    \geq|x_{N+2}|.
\end{equation}
Then, by induction, it follows that $|x_{n+1}|\geq|x_{n}|$ for all $n>N$.
\end{proof}

\subsection{Properties of the Chebyshev polynomials} \label{sec:polynomials}

\begin{figure}
    \centering
    \includegraphics[width=0.6\linewidth]{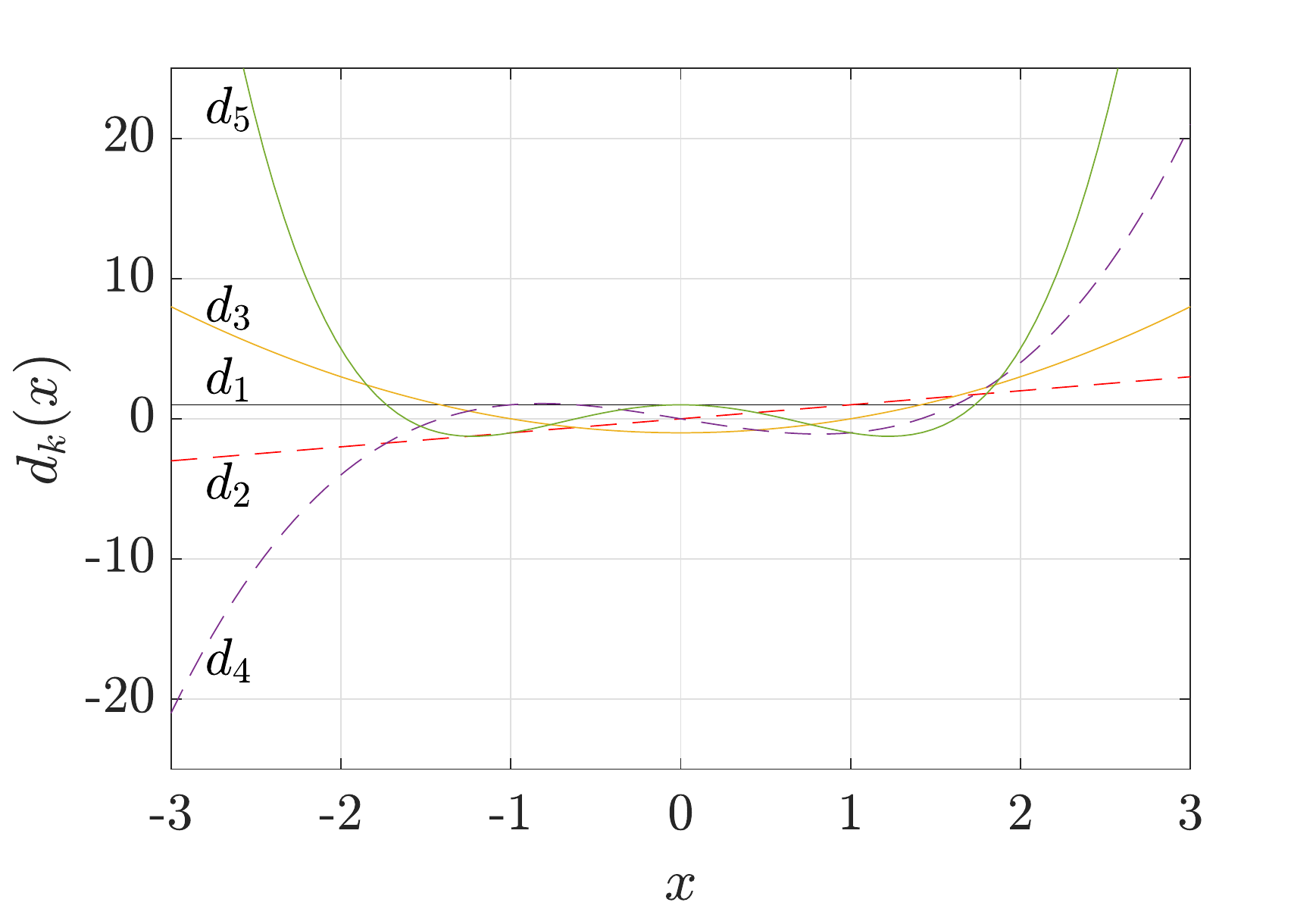}
    \caption{The first few Chebyshev polynomials $d_k(x)$, defined recursively in \eqref{defn:dk}. These functions play a crucial role in determining the behaviour of generalised Fibonacci tilings.}
    \label{fig:dkplot}
\end{figure}

Before proceeding to study super band gaps in more exotic generalised Fibonacci tilings, we must first prove some properties of the polynomials $d_k(x)$ defined in \eqref{defn:dk}. The first few $d_k(x)$ are plotted in Figure~\ref{fig:dkplot}, for reference. Using standard techniques (as in \emph{e.g.} Section~2.3 of \cite{bender1999advanced}), we can derive an explicit formula for $d_k(x)$, which is given by
\begin{equation} \label{eq:dkformula}
    d_k(x)=\frac{1}{\sqrt{x^2-4}} \left( \frac{x+\sqrt{x^2-4}}{2} \right)^k
    - \frac{1}{\sqrt{x^2-4}} \left( \frac{x-\sqrt{x^2-4}}{2} \right)^k
\end{equation}
for $k\in\mathbb{Z}^{\geq0}$ and $x\in(0,\infty)\setminus\{2\}$. To check the value of the solution at $x=2$, we have the following lemma:

\begin{lemma} \label{lem:dk2=k}
    $d_k(2)=k$ for all $k\geq 0$.
\end{lemma}
\begin{proof}
    This is true, by definition, for $k=0$ and $k=1$. If we suppose that it holds for arbitrary $k-1$ and $k$ then we have that
    \begin{equation}
        d_{k+1}(2)=2d_{k}(2)-d_{k-1}(2)=2k-(k-1)=k+1,
    \end{equation}
    so the result follows by induction on $k$.
\end{proof}

The definition \eqref{defn:dk}, alongside the formula \eqref{eq:dkformula}, can be used to study the properties of the sequence of polynomials. For example, it will be helpful to understand the parity of $d_k$:

\begin{lemma} \label{lem:parity}
    For $k\geq1$, if $k$ is odd then $d_k(x)$ contains only even powers of $x$ and if $k$ is odd then $d_k(x)$ contains only even powers of $x$.
\end{lemma}
\begin{proof}
    We can immediately check the first few terms: $d_1(x)=1$, $d_2(x)=x$, $d_3(x)=x^2-1$. Then, we suppose that the statement holds true for some $k$ and $k-1$, where $k$ is even. In which case $xd_k(x)$ contains only odd powers of $x$, meaning that $d_{k+1}(x) = x d_{k}(x)-d_{k-1}(x)$ contains only odd powers. A similar argument holds for odd $k$. The result follows by induction.
\end{proof}

A consequence of Lemma~\ref{lem:parity} is that $d_k$ is an even function when $k$ is odd and is an odd function when $k$ is even. This means it is sufficient to study its properties when $x>0$. We have the following results, which will allow us to derive bounds on these polynomials when $|x|>2$ (which is the domain of interest).

\begin{lemma} \label{lem:dkpositive}
    $d_k(x)\geq0$ and $d_k'(x)\geq0$ for all $k\geq 0$ and all $x\geq 2$, with equality holding only if $k=0$.
\end{lemma}
\begin{proof}
    This is trivial for $k=0$, so we consider $k\geq1$. From Lemma~\ref{lem:dk2=k}, we have that $d_k(2)=k>0$ for all $k\geq 1$. For $x>2$, it holds that $x+\sqrt{x^2-4}>x-\sqrt{x^2-4}>0$. Thus, since $x\mapsto x^k$ is strictly increasing for $x\geq0$, it follows that 
    \begin{equation}
        \left(x+\sqrt{x^2-4}\right)^k-\left(x-\sqrt{x^2-4}\right)^k>0.
    \end{equation}
    So, using the formula \eqref{eq:dkformula}, we find that $d_k(x)>0$ for $k\geq1$ and $x>2$.

    To handle the derivative, we notice that $d_k(x)$ is the determinant of the $k\times k$ tridiagonal matrix $M_k(x)$ given by
    \begin{equation}
        M_k(x)_{ij}=\begin{cases}
            x & \text{if } i=j,\\
            1 & \text{if } i-j=\pm1,\\
            0 & \text{otherwise}.
        \end{cases}
    \end{equation}
    Since $d_k(x)>0$ for $k\geq1$ and $x\geq2$, $M_k(x)$ must be invertible. Hence, we can use Jacobi's formula to see that
    \begin{equation} \label{eq:jacobi}
        \frac{\mathrm{d}}{\mathrm{d}x}d_k(x) = \frac{\mathrm{d}}{\mathrm{d}x} \det(M_k(x)) 
        = \det(M_k(x)) \,\tr\bigg( M_k(x)^{-1} \frac{\mathrm{d}}{\mathrm{d}x}M_k(x) \bigg)
        = d_k(x) \,\tr( M_k(x)^{-1}),
    \end{equation}
     where we have used the fact that the derivative of $M_k(x)$ with respect to $x$ is the identity matrix. 
     
    To deal with $\tr( M_k(x)^{-1})$, we will show that $M_k(x)$ has strictly positive eigenvalues whenever $k\geq1$ and $x\geq2$. For $x>2$, this follows immediately from the Gershgorin circle theorem. When $x=2$, Gershgorin circle theorem permits eigenvalues to vanish, but this is forbidden by the invertibility of $M_k(x)$. Thus, if $k\geq1$ and $x\geq2$, then $M_k(x)$ has strictly positive eigenvalues $\lambda_1(x),\dots,\lambda_k(x)$. Finally, using the fact that $M_k(x)$ is symmetric and positive definite, we can compute that
     \begin{equation} \label{eq:trace}
         \tr( M_k(x)^{-1})=\sum_{i=1}^k \lambda_i(x)^{-1} >0.
     \end{equation}
     Combining this with the fact that $d_k(x)>0$, \eqref{eq:jacobi} tells us that $d_k'(x)>0$ for all $k\geq1$ and $x\geq2$.
\end{proof}

\begin{cor} \label{cor:dk>2}
    $|d_k(x)|\geq2$ for all $k\geq2$ and all $|x|\geq2$.
\end{cor}
\begin{proof}
    This follows by combining Lemma~\ref{lem:dkpositive} with Lemma~\ref{lem:dk2=k}, for $x>2$. Then, the result for $x<-2$ follows by parity.
\end{proof}

\begin{lemma}
    $d_{k+1}(x)\geq d_k(x)$ for all $k\geq0$ and all $x\geq2$.
\end{lemma}
\begin{proof}
    This is true for $k=0$, from the definition. Then, supposing that $d_{k}(x)\geq d_{k-1}(x)$, 
    \begin{equation}
        d_{k+1}(x)=xd_k(x)-d_{k-1}(x)
        \geq 2d_k(x)-d_{k-1}(x) = d_k(x)+(d_k(x)-d_{k-1}(x))
        \geq d_k(x),
    \end{equation}
    where the first inequality relies on the fact that $d_k(x)\geq0$ from Lemma~\ref{lem:dkpositive}. Finally, the result follows by induction on $k$.
\end{proof}

Using the odd/even parity of the polynomials $d_k$, we have the following corollary:

\begin{cor} \label{lem:dkgrowth}
    $|d_{k+1}(x)|\geq |d_k(x)|$ for all $k\geq0$ and all $|x|\geq2$.
\end{cor}

The final property of the polynomials $d_k(x)$ that we will need is the following inequality:

\begin{lemma} \label{lem:newgrowth}
    $|d_{k+1}(x)| \leq |xd_{k}(x)| \leq 2 |d_{k+1}(x)|$ for any $|x|>2$  and any $k\geq1$.
\end{lemma}
\begin{proof}
    Thanks to the parity of $d_k$, we can consider $x>2$ without loss of generality, in which case $d_k(x)\geq0$ for all $k$. For the first inequality, we have that
    \begin{equation}
        0\leq d_{k-1}(x) = xd_{k}(x) - d_{k+1}(x),
    \end{equation}
    so $d_{k+1}(x)\leq xd_{k}(x)$. To see the second inequality, we must use the formula \eqref{eq:dkformula}. It holds that 
    \begin{equation}
        d_{k+1}(x)=\frac{1}{\sqrt{x^2-4}} \frac{x+\sqrt{x^2-4}}{2}\left( \frac{x+\sqrt{x^2-4}}{2} \right)^k
    - \frac{1}{\sqrt{x^2-4}} \frac{x-\sqrt{x^2-4}}{2} \left( \frac{x-\sqrt{x^2-4}}{2} \right)^k.
    \end{equation}
    We have that $x+\sqrt{x^2-4}\geq x$ and $-(x-\sqrt{x^2-4})\geq-x$, from which we see that $d_{k+1}(x)\geq \frac{1}{2}x d_k(x)$.
\end{proof}

\subsection{Generalised precious mean Fibonacci}

Generalised Fibonacci tilings with $l=1$ and arbitrary $m$ are known as \emph{precious mean} Fibonacci tilings (generalizing the notions of golden and silver means for $m=1$ and $m=2$, respectively). In this case the recursion relation \eqref{rec:general} reads
\begin{equation} \label{rec:precious}
    \begin{cases}
        x_{n+1}=d_m(x_n)t_{n+1}-d_{m-1}(x_n)x_{n-1},\\
        t_{n+1}=d_{m+1}(x_{n-1})t_n-d_m(x_{n-1})x_{n-2}.
    \end{cases}
\end{equation}
for $n\geq 2$. In order to develop a precise theory for super band gaps when $m>2$, we will need to assume that the sequence of traces has at least polynomial growth, with order $m-1$. This is consistent with the rule that was established for the silver mean in Theorem~\ref{thm:silver}. In fact, we will need that terms grow such that $|x_{n+1}|\geq |d_{m-1}(x_n)x_n|$. This is made precise by the following theorem.

\begin{thm} \label{thm:precious}
    Let $\omega\in\mathbb{R}$ and consider $x_n(\omega)$ satisfying the generalised precious mean recursion relation \eqref{rec:precious} for some $m\geq2$. Suppose that there exists some $N\in\mathbb{N}$ such that
    \begin{equation*}
        |x_{N}|>2, \quad
        |x_{N+1}|\geq|d_{m-1}(x_N)x_{N}|\quad\text{and}\quad
        |x_{N+2}|\geq|d_{m-1}(x_{N+1})x_{N+1}|.
    \end{equation*}
    Then $|x_{n+1}|\geq|d_{m-1}(x_n)x_{n}|$ for all $n>N$.
    Consequently, $|x_n|>2$ for all $n\geq N$, meaning that $\omega$ is in the super band gap $\S_N$.
\end{thm}
\begin{proof}
The special case $m=2$ is exactly the result that was proved in Theorem~\ref{thm:silver}, since $d_1(x)=1$. We will consider $m\geq3$. We begin by rewriting the recursion relation \eqref{rec:precious} in this case. From the first equation of \eqref{rec:precious}, we have that 
\begin{equation} \label{eq:rearrange1}
    d_m(x_{n-1})t_{n}=x_{n}+d_{m-1}(x_{n-1})x_{n-2}.
\end{equation}
Turning to the second equation of \eqref{rec:precious}, using the definition of $d_k$ and substituting \eqref{eq:rearrange1} gives
\begin{align} \label{eq:rearranget}
    t_{n+1}= x_{n-1}x_n +d_{m-2}(x_{n-1})x_{n-2} -d_{m-1}(x_{n-1})t_n.
\end{align}

An important observation is that, thanks to Corollary~\ref{cor:dk>2}, the hypotheses of this theorem imply that $|x_{N+2}|\geq|x_{N+1}|\geq|x_{N}|>2$. This is important as $\omega$ could not be in the super band gap $\S_N$ otherwise. It also allows us to use the inequality \eqref{eq:tn1} to see that 
\begin{equation}
    |t_{N+2}| \leq x_{N+1}^2-2.
\end{equation}
Then, from \eqref{eq:rearranget}, we have that 
\begin{align}
    |t_{N+3}|&\geq |x_{N+1}x_{N+2}| -|d_{m-2}(x_{N+1})x_{N}| -|d_{m-1}(x_{N+1})t_{N+2}| \nonumber \\
    &\geq |x_{N+1}x_{N+2}| -|d_{m-2}(x_{N+1})x_{N}| +2|d_{m-1}(x_{N+1})| -|x_{N+1}^2d_{m-1}(x_{N+1})| \nonumber \\
    &\geq -|d_{m-2}(x_{N+1})x_{N}| +2|d_{m-1}(x_{N+1})|, \label{eq:tN3}
\end{align}
where the last inequality follows by hypothesis.

To deal with \eqref{eq:tN3}, we must turn to Lemma~\ref{lem:newgrowth}. Since $|x_N|>2$, $|d_{m-1}(x_N)|\geq |d_2(x_N)|\geq d_2(2)=2$. As a result, the assumption that $|x_{N+1}|\geq|x_{N} d_{m-1}(x_N)|$ implies that $|x_{N+1}|\geq2|x_{N}|> 4$. Consequently, we have that
\begin{equation}
    |d_{m-2}(x_{N+1})x_{N}|\leq \frac{1}{2}|d_{m-2}(x_{N+1})x_{N+1}|\leq |d_{m-1}(x_{N+1})|.
\end{equation}
Using this inequality, \eqref{eq:tN3} gives us that 
\begin{equation}
    |t_{N+3}|\geq  |d_{m-1}(x_{N+1})|
    \geq d_{2}(4)
    = 4.
\end{equation}

We can now turn to the first equation of \eqref{rec:precious}, which gives us that
\begin{equation}
    |x_{N+3}|\geq|d_m(x_{N+2})t_{N+3}|-|d_{m-1}(x_{N+2})x_{N+1}| 
    \geq 4|d_m(x_{N+2})|-|d_{m-1}(x_{N+2})x_{N+1}|.
\end{equation}
Using Lemma~\ref{lem:newgrowth} again, we have that
\begin{equation}
    |x_{N+3}|
    \geq 2|d_{m-1}(x_{N+2})x_{N+2}|-|d_{m-1}(x_{N+2})x_{N+1}|
    \geq |d_{m-1}(x_{N+2})x_{N+2}|.
\end{equation}
where the second inequality follows from the fact that $|x_{N+2}|\geq|x_{N+1}|$. Proceeding by induction gives us that $|x_{n+1}|\geq|d_{m-1}(x_n)x_{n}|$ for all $n>N$. Thanks to Corollary~\ref{cor:dk>2}, we see also that $|x_n|\geq|x_N|>2$ for all $n\geq N$.
\end{proof}

\subsection{Generalised metal mean Fibonacci}

Suppose now that $m=1$ and $l$ is arbitrary. This case is sometimes known as the \emph{metal mean} generalised Fibonacci. In  particular, $l=2$ is known as the \emph{copper mean} and $l=3$ as the \emph{nickel mean} \cite{gumbs1988dynamical, kolar1989attractors}. In this case, we are able to eliminate $t_n$ from the recursion relation \eqref{rec:general}, giving the simpler recursion relation
\begin{equation} \label{rec:metal}
    x_{n+1} = d_l(x_{n-1}) [x_n x_{n-1} - d_{l+1}(x_{n-2}) + d_{l-1}(x_{n-2})] - x_n d_{l-1}(x_{n-1}). 
\end{equation}
for $n\geq 2$. Notice how this reduces to the golden mean recursion relation \eqref{rec:golden} in the case that $l=1$.

\begin{thm} \label{thm:metal}
    Let $\omega\in\mathbb{R}$ and consider $x_n(\omega)$ satisfying the generalised metal mean recursion relation \eqref{rec:metal} for some $l\geq1$. Suppose that there exists some $N\in\mathbb{N}$ such that
    \begin{equation*}
        |x_{N}|>2, \quad
        |x_{N+1}|\geq \frac{5}{2}\quad\text{and}\quad
        |x_{N+2}|\geq\max\{|x_{N+1}|,|d_{l+1}(x_{N})|\}.
    \end{equation*}
    Consequently, $|x_n|>2$ for all $n\geq N$, meaning that $\omega$ is in the super band gap $\S_N$.
\end{thm}
\begin{proof}
    The special case $l=1$ was proved in Theorem~\ref{thm:gold}. For $l\geq2$, we have from \eqref{rec:metal} that
    \begin{equation} \label{eq:metalineq}
        |x_{N+3}|\geq |x_{N+1}x_{N+2}d_l(x_{N+1})| - |d_l(x_{N+1})[d_{l+1}(x_N) - d_{l-1}(x_N)]| - |x_{N+2}d_{l-1}(x_{N+1})|.
    \end{equation}
    
    We know that $|d_{l+1}(x_N)|\geq|d_{l-1}(x_N)|$ and they must both have the same sign since they have the same parity and do not vanish on $|x_N|>2$. As a result, we have that 
    \begin{equation}
        |d_{l+1}(x_N) - d_{l-1}(x_N)| = |d_{l+1}(x_N)| - |d_{l-1}(x_N)| \leq |d_{l+1}(x_N)| \leq |x_{N+2}|,
    \end{equation}
    where the final inequality follows by hypothesis. Substituting this into \eqref{eq:metalineq} gives 
    \begin{align}
        |x_{N+3}|&\geq |x_{N+1}x_{N+2}d_l(x_{N+1})| - |x_{N+2} d_l(x_{N+1})| - |x_{N+2}d_{l-1}(x_{N+1})| \nonumber \\
        &\geq \left( |x_{N+1}| - 2 \right) | |x_{N+2} d_l(x_{N+1})|. \label{eq:metalineq2}
    \end{align}
    Since $|x_{N+1}|<|x_{N+2}|$, we can use Lemma~\ref{lem:newgrowth} to see that $|x_{N+2} d_l(x_{N+1})|\geq |x_{N+1} d_l(x_{N+1})|\geq |d_{l+1}(x_{N+1})|$. Since $|x_{N+1}|-2>0$, we conclude that 
    \begin{equation}
        |x_{N+3}|\geq |d_{l+1}(x_{N+1})|.
    \end{equation}
    We also need to check that $|x_{N+3}|\geq |x_{N+2}|$. This follows from \eqref{eq:metalineq2} since $|d_l(x_{N+1})|\geq d_l(2)=l\geq 2$ and $|x_{N+1}| - 2\geq \frac{1}{2}$. 
    
    Finally, we can proceed by induction to see that $|x_{n+2}|\geq\max\{|x_{n+1}|,|d_{l+1}(x_{n})|\}$ for all $n\geq N$. Since $|x_{n+1}|\geq \frac{5}{2}$ for all $n\geq N$, it follows that $|x_n|>2$ for all, so it must it hold that $\omega\in \S_N$.
\end{proof}

\subsection{Discussion}

We have established a new theory for super band gaps, which characterises when the sequence of traces $x_n(\omega)$ is guaranteed to grow indefinitely. A natural question to ask of the results proved in this section is whether the growth conditions are optimal. In the case of Theorems~\ref{thm:gold} and~\ref{thm:silver}, the results for the golden and silver mean tilings respectively, the simple growth condition is likely to be the strongest possible result. However, this is less clear for the other generalised Fibonacci tilings. In particular, we suspect that Theorem~\ref{thm:metal}, the result for generalised metal mean Fibonacci tilings, could be improved. The requirement that $|x_{N+1}|\geq 5/2$, for example, is almost certainly not optimal. We used this assumption to derive one of the bounds needed for the inductive hypothesis, however it is likely that this assumption could be relaxed by future work. Nevertheless, the numerical evidence we will present in Section~\ref{sec:examples} demonstrates that even this sub-optimal result still gives a precise prediction of the super band gaps (we will present numerical results for the copper mean tiling for each physical system). The reason for this is that within these super band gaps (particularly away from the edges) the sequence of traces $x_n(\omega)$ typically grows very quickly, so the sub-optimality of the growth condition has little effect. This very rapid growth in the middle of super band gaps is also the reason that the estimator $H_2(\omega)$, defined in \eqref{eq:Hdefn} and introduced by \cite{morini2018waves}, performed relatively well at predicting their approximate locations.

\section{Super band gaps in specific one-dimensional systems} \label{sec:examples}

The general theory from the previous section can be applied to study the spectral properties of generalised Fibonacci tilings in various one-dimensional systems. We will consider three different examples: a discrete mass-spring system, a structured rod and a continuous beam with modulated distances between the supports.

\subsection{Compressional waves in discrete mass-spring systems}
As a first example, we consider a periodic discrete mass-spring system. The fundamental cells are designed according to the generalised Fibonacci substitution rule \eqref{eq:tiling}, where the two elements $A$ and $B$ correspond to different masses $m_A$ and $m_B$ and linear springs with stiffness $k_A$ and $k_B$, respectively (see Figure~\ref{fig:systems}$/(a)$). In order to study the dispersive properties of harmonic compressional waves in this system, we study the horizontal displacement of each mass $u_j(t)=u_je^{i\omega t}$ and the harmonic force acting on that mass $f_j(t)=f_je^{i\omega t}$, where the index $j$ indicates the relevant mass. Thus, we introduce the state vector in the frequency domain $\mathbf{u}_j= [u_j, f_j]^T$. The relationship between $\mathbf{u}_j$ and the state vector of the preceding element $\mathbf{u}_{j-1}$ is given by \cite{Lazaro2022Sturm}:
\begin{equation} \label{eq:massspring}
\mathbf{u}_j=\left[
\begin{array}{c}
u_j \\
f_j
\end{array}
\right]=
\left[
\begin{array}{cc}
1  & -\cfrac{1}{k_X} \\
m_X\omega^2 & 1-\cfrac{m_X\omega^2}{k_X} \\
\end{array}
\right]
\left[
\begin{array}{c}
u_{j-1} \\
f_{j-1}
\end{array}
\right]\equiv T^X(\omega, m_X, k_X)\mathbf{u}_{j-1}, \quad \mathrm{with} \quad X \in \left\{A, B\right\}.
\end{equation}
$T^X(\omega, m_X, k_X)$ is the transfer matrix of a single element $A$ or $B$, and corresponds to the product of the respective transfer matrices associated with the mass $m_X$ and the spring of stiffness $k_X$ \cite{RuiWang2019}. 

Given a generalised Fibonacci unit cell $\F_n$, the state vector $\mathbf{u}_{F_{n}}$ at the right-hand boundary of the unit cell (corresponding to $j=F_{n}$ where $F_{n}$ is the previously defined  generalised Fibonacci number) can then be expressed in terms of the state vector at the left-hand boundary, $\mathbf{u}_{0}$, according to
\begin{equation}
\mathbf{u}_{F_{n}}=T_n(\omega)\mathbf{u}_{0},
\label{uFn}
\end{equation}
where $T_n(\omega)=\Pi_{p=1}^{F_{n}}T^X(\omega, m_X, k_X)$ is the transfer matrix of the fundamental cell of order $n$. Applying the Floquet-Bloch theorem to the unit cell, we substitute $\mathbf{u}_{F_{n}}=\mathbf{u}_0e^{iKL_n}$ into equation (\ref{uFn}), and due to the fact that $T_n(\omega)$ is endowed with the unimodularity and recursive properties illustrated in Section \ref{gentheory}, the dispersion relation takes the form 
\begin{equation}
\cos({KL_n})=\frac{1}{2}\tr( T_n({\omega})) \quad \Rightarrow \quad KL_{n}=\arccos\left(\frac{\tr( T_{n}(\omega))}{2}\right),
\label{disp}
\end{equation}
where $L_n$ is the length of the unit cell.

\begin{figure}
    \centering
    \begin{subfigure}{0.475\textwidth}
        \includegraphics[width=\linewidth,trim=1cm 0 1.5cm 0,clip]{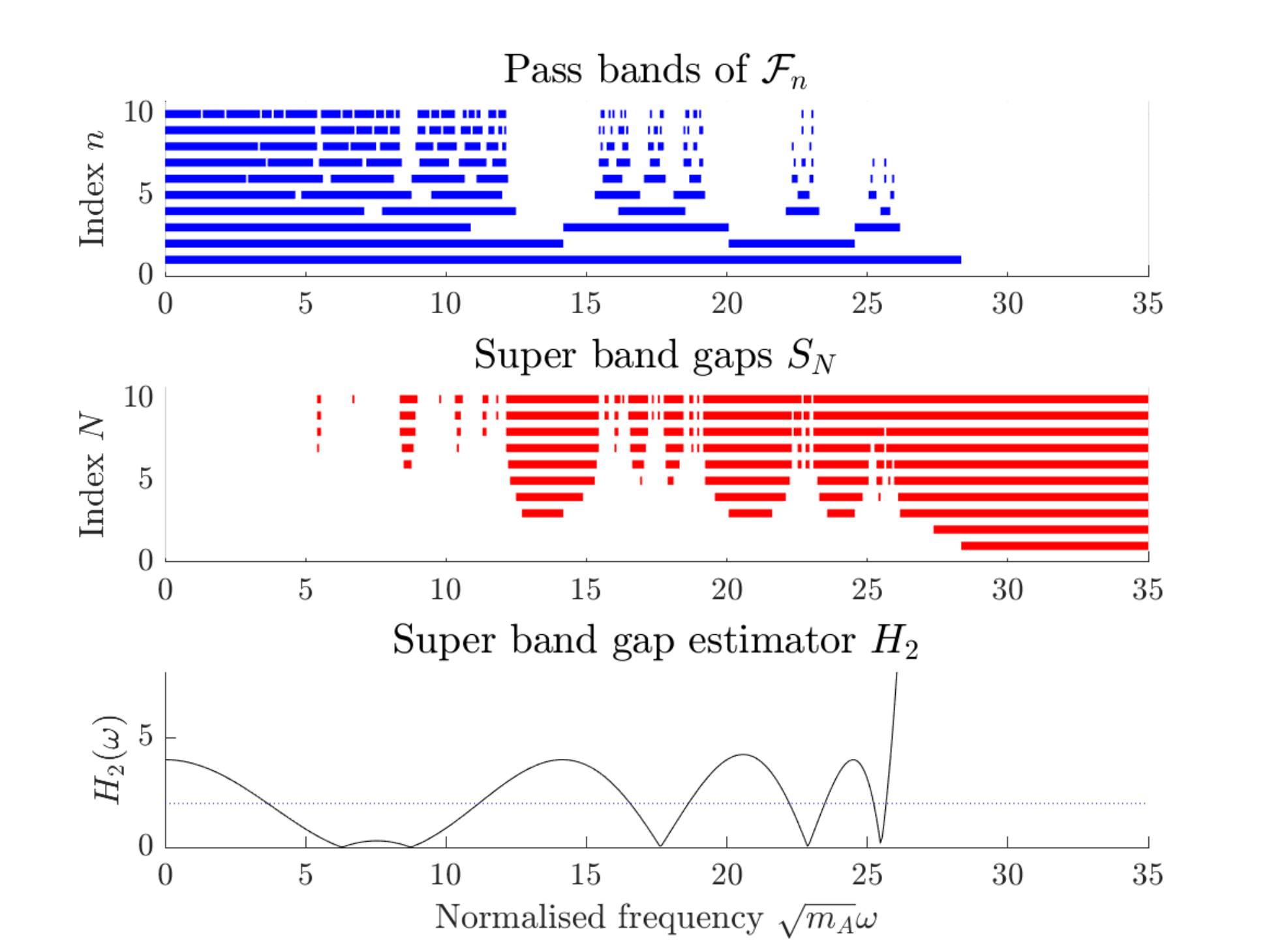}
        \caption{Golden mean Fibonacci ($m=1$, $l=1$).}
    \end{subfigure}
    \hfill
    \begin{subfigure}{0.475\textwidth}
        \includegraphics[width=\linewidth,trim=1cm 0 1.5cm 0,clip]{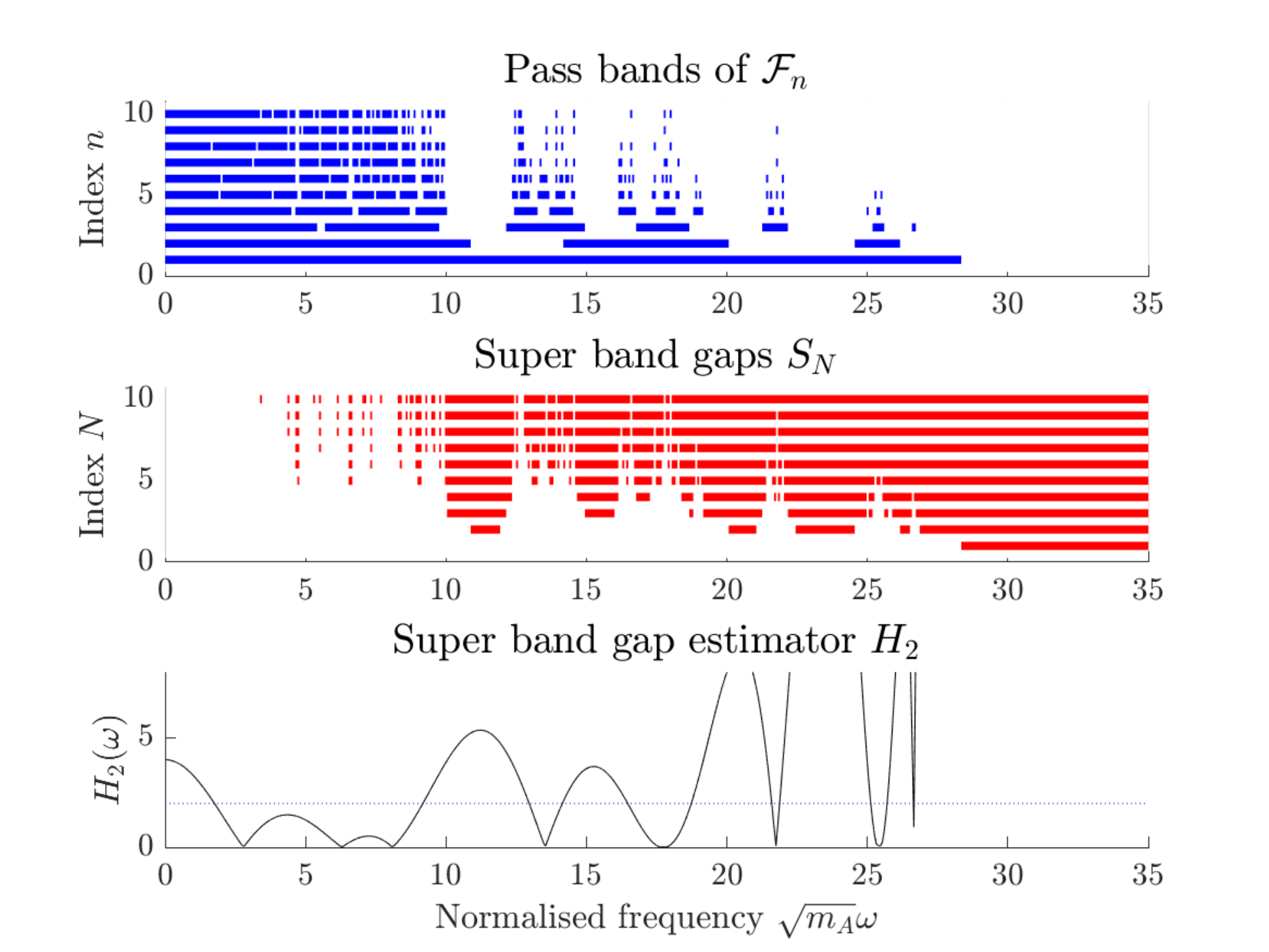}
        \caption{Silver mean Fibonacci ($m=2$, $l=1$).}
    \end{subfigure}

    \vspace{0.2cm}

    \begin{subfigure}{0.475\textwidth}
        \includegraphics[width=\linewidth,trim=1cm 0 1.5cm 0,clip]{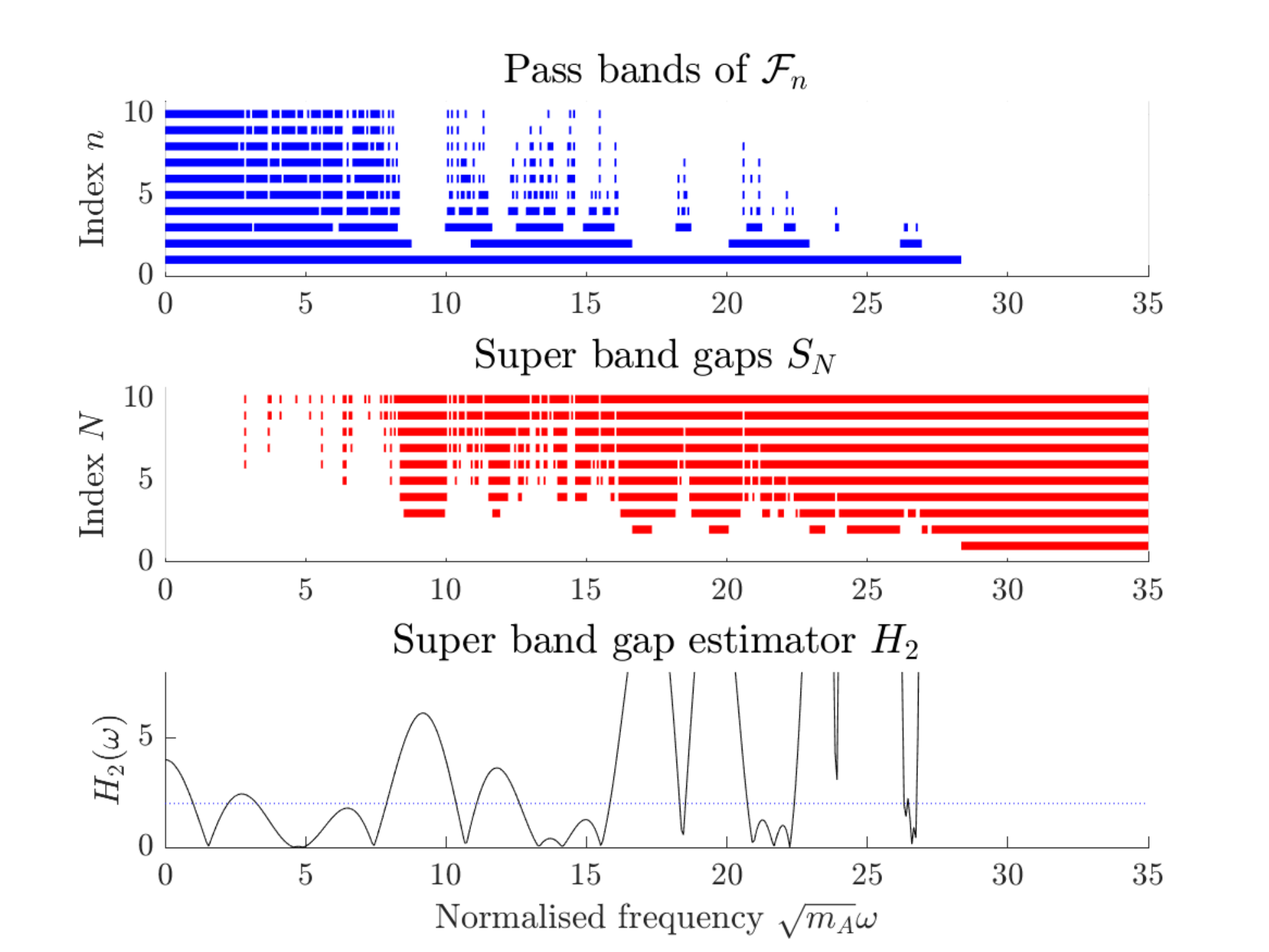}
        \caption{Bronze mean Fibonacci ($m=3$, $l=1$).}
    \end{subfigure}
    \hfill
    \begin{subfigure}{0.475\textwidth}
        \includegraphics[width=\linewidth,trim=1cm 0 1.5cm 0,clip]{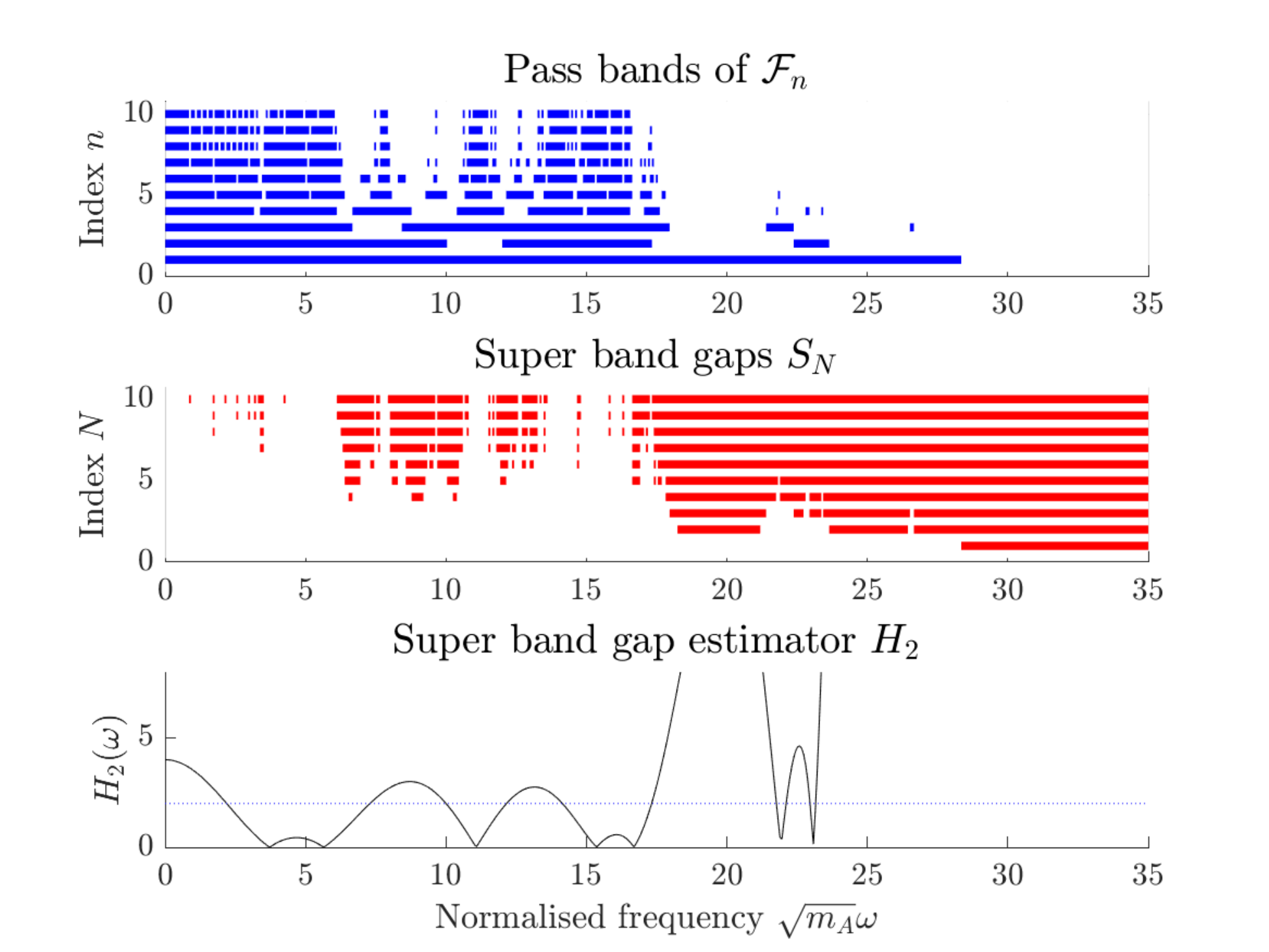}
        \caption{Copper mean Fibonacci ($m=1$, $l=2$).}
    \end{subfigure}
    
    \caption{The pass bands and super band gaps of a discrete mass-spring system with spring constants varied according to generalized Fibonacci tilings $\F_n$. For each tiling, the top plot shows the pass bands for each successive Fibonacci tiling $\F_n$, the middle shows the super band gaps $\S_n$, as predicted by the corresponding theorem, and the bottom shows the super band gap estimator $H_2$, as used in previous works and defined in \eqref{eq:Hdefn}. We use the parameter values $k_A=2k_B=200$N/m and suppose that $m_A=m_B$. The normalised frequency $\sqrt{m_A}\omega$ is shown on the horizontal axes.
    }
    \label{fig:SBG_massspring}
\end{figure}

The pattern of pass and stop bands for this discrete mass-spring system is shown in Figure~\ref{fig:SBG_massspring} for various generalised Fibonacci tiling. In each case, the upper plot shows the pass bands of successive tilings $\F_n$, characterised as $\omega$ such that $|\tr(T_n(\omega))|\leq2$. We can see how the spectrum becomes increasingly complex for increasing $n$. The middle plots of Figure~\ref{fig:SBG_massspring} show the super band gaps $\S_N$, which are computed by checking if $\tr(T_N(\omega))$ satisfies the growth condition from the theorems in Section~\ref{gapstheory}. We see that the super band gaps agree with the pattern of spectral gaps observed in the top plot. By looking at the super band gaps $\S_N$ for larger $N$, our theory is able to reveal some of the complex structure that emerges for $\F_n$ with large $n$ and shows that many of the smaller band gaps that are created are, in fact, super band gaps. 

The lower plots in Figure~\ref{fig:SBG_massspring} show the super band gap estimator function $H_2(\omega)=|\tr(T_2(\omega))\tr(T_3(\omega))|$ from \cite{morini2018waves}. We can see that the local maxima of $H_2$ successfully predict the locations of the super band gaps, but that it is unable to reveal the complex spectral behaviour that occurs for higher-order Fibonacci tilings. This shows another benefit of our new theory over the existing methods. We have not only developed a rigorous theory for super band gaps, but our theory has greater resolution than was previously possible.

One notable feature of Figure~\ref{fig:SBG_massspring} is the occurrence of high-frequency super band gaps. That is, there appears to exist some $\omega^*$ such that any $\omega>\omega^*$ is in a super band gap. The origin for this phenomenon can be seen by inspecting the transfer matrices $T^A$ and $T^B$, defined in \eqref{eq:massspring}. We have that 
\begin{equation}
    \tr(T^X(\omega, m_X, k_X))= 2 - \cfrac{m_X\omega^2}{k_X},
\end{equation}
so it is easy to see that if $\omega>2\sqrt{k_X/m_X}$ then $\tr(T^X)<-2$ so $\omega$ is in a band gap of the material with label $X$. As a result, we have that if $\omega>\max\left\{ 2\sqrt{k_A/m_A}, 2\sqrt{k_B/m_B} \right\}$ then $\omega$ is in band gaps of both $\F_0$ and $\F_1$, for any generalised Fibonacci tiling. However, this is not generally enough to guarantee that $\omega$ is in a super band gap. For the discrete mass-spring system, the super band gap occurs due to the structure of the associated transfer matrices, which take a specific form when $\omega$ is sufficiently large. This is made precise with the following result.

\begin{thm} \label{thm:highfreqgaps}
    Consider a discrete mass-spring system with behaviour governed by the equation \eqref{eq:massspring} and fundamental cells designed according to a generalised Fibonacci substitution rule \eqref{eq:tiling} with arbitrary $m,l\geq1$. There exists some $\omega^*$ such that if $\omega>\omega^*$ then $\omega$ is in the super band gap $\S_0$.
\end{thm}
\begin{proof}
Suppose that $\omega\to\infty$ while all the other parameters are kept constant. In this case, we have that 
\begin{equation}
    T^X = m_X\omega^2\left(\left[
\begin{array}{cc}
0  & 0 \\
1 & -k_X^{-1} 
\end{array}
\right]
+O(\omega^{-2})\right)
\quad\text{as}\quad \omega\to\infty.
\end{equation}
Then, some straightforward algebra reveals that the transfer matrix of the generalised Fibonacci tiling $\F_n$ satisfies
\begin{equation}
    T_n = (m_A)^{mF_{n-2}}(m_B)^{lF_{n-1}}\omega^{2F_n}\left(\left[
\begin{array}{cc}
0  & 0 \\
\eta_1 & \eta_2 
\end{array}
\right]
+O(\omega^{-2})\right)
\quad\text{as}\quad \omega\to\infty,
\end{equation}
where $\eta_1$ and $\eta_2$ are non-zero constants and the generalised Fibonacci numbers $F_n$ were defined in \eqref{eq:Fibonaccinumbers}. Crucially, it holds that $|\eta_2|\geq \max\{k_A,k_B\}^{-F_n}$, so we can see that
\begin{equation}
    |\tr(T_n)|\geq \frac{(m_A)^{mF_{n-2}}(m_B)^{lF_{n-1}}}{\max\{k_A,k_B\}^{F_n}}\omega^{2F_n}
    \geq \left( \frac{\min\{m_A,m_B\}}{\max\{k_A,k_B\}}\omega^{2}\right)^{F_{n}}.
\end{equation}
As a result, we can see that if $\omega$ is sufficiently large, then $|\tr(T_n)|>2$ for all $n$, implying that $\omega$ is in the super band gap $S_0$.
\end{proof}

\subsection{Axial waves in structured rods} \label{sec:rods}
The dispersive properties of two-phase quasiperiodic structured rods with unit cells generated by one-dimensional generalised Fibonacci sequences have been studied previously in \cite{morini2018waves}, including experimentally by \cite{gei2020phononic}. The lengths of the two segments $A$ and $B$ are indicated with $l_{A}$ and $l_{B}$, respectively, while $A_X$, $E_X$ and $\rho_X$ denote the cross-sectional area, Young's modulus and mass density per unit volume of the two adopted materials, respectively. This is sketched in Figure~\ref{fig:systems}$/(b)$. For both elements, we define the displacement function and the axial force along the rod as $u(z)$ and $N(z)=EAu^{'}(z)$, respectively, where $z$ is the coordinate describing the longitudinal axis (as depicted in Figure~\ref{fig:systems}). The governing equation of harmonic axial waves in each section is given by 
\begin{equation}
u^{''}_X(z)+Q_X\omega^2 u_X(z)=0,
\label{axialwave}
\end{equation}
where $Q_X=\rho_X/E_X$ corresponds to the reciprocal of the square of the speed of propagation of
longitudinal waves in material $X$. The general solution of \eqref{axialwave} is given by
\begin{equation}
u_X(z)=C_1^X\sin\left(\sqrt{Q_X}\omega z\right)+C_2^X\cos\left(\sqrt{Q_X}\omega z\right),
\end{equation}
where $C_1^X$ and $C_2^X$ are integration constants, to be determined by the boundary conditions.

In order to obtain the dispersion diagram of the quasiperiodic rod, we express the state vector $\mathbf{u}_{F_n}=[u_{F_n}, N_{F_n}]^T$ at the end of the Fibonacci unit cell as a function of the same vector $\mathbf{u}_0=[u_0, N_0]^T$ on the left-hand side:
\begin{equation}
\mathbf{u}_{F_n}=T_n(\omega)\mathbf{u}_0,
\end{equation}
where $T_n(\omega)$ is a transfer matrix of the cell $\mathcal{F}_n$. This matrix is the result of the product  $T_n(\omega)=\Pi_{p=1}^{F_{n}}T^X(\omega)$, where $T^X(\omega)$ ($X\in \left\{A, B\right\}$) is the transfer matrix which relates quantities across a single element, given by
\begin{equation}
T^X(\omega)=
\left[
\begin{array}{cc}
\cos\left(\sqrt{Q_X}\omega l_X\right) & \cfrac{\sin\left(\sqrt{Q_X}\omega l_X\right)}{E_X A_X\sqrt{Q_X}\omega}\\
                                      &                                \\
-E_X A_X \omega \sqrt{Q_X}\sin\left(\sqrt{Q_X}\omega l_X\right) & \cos\left(\sqrt{Q_X}\omega l_X\right)\\
\end{array}
\right].
\label{Tmrod}
\end{equation}
Once again, the matrices $T_n(\omega)$ possess the important properties introduced in Section \ref{gentheory}. As a consequence, if we impose the Floquet-Bloch condition $\mathbf{u}_r=\mathbf{u}_l e^{iKL_n}$, then the corresponding dispersion relation assumes a form identical to (\ref{disp}).

\begin{figure}
    \centering
    \begin{subfigure}{0.475\textwidth}
        \includegraphics[width=\linewidth,trim=1cm 0 1.5cm 0,clip]{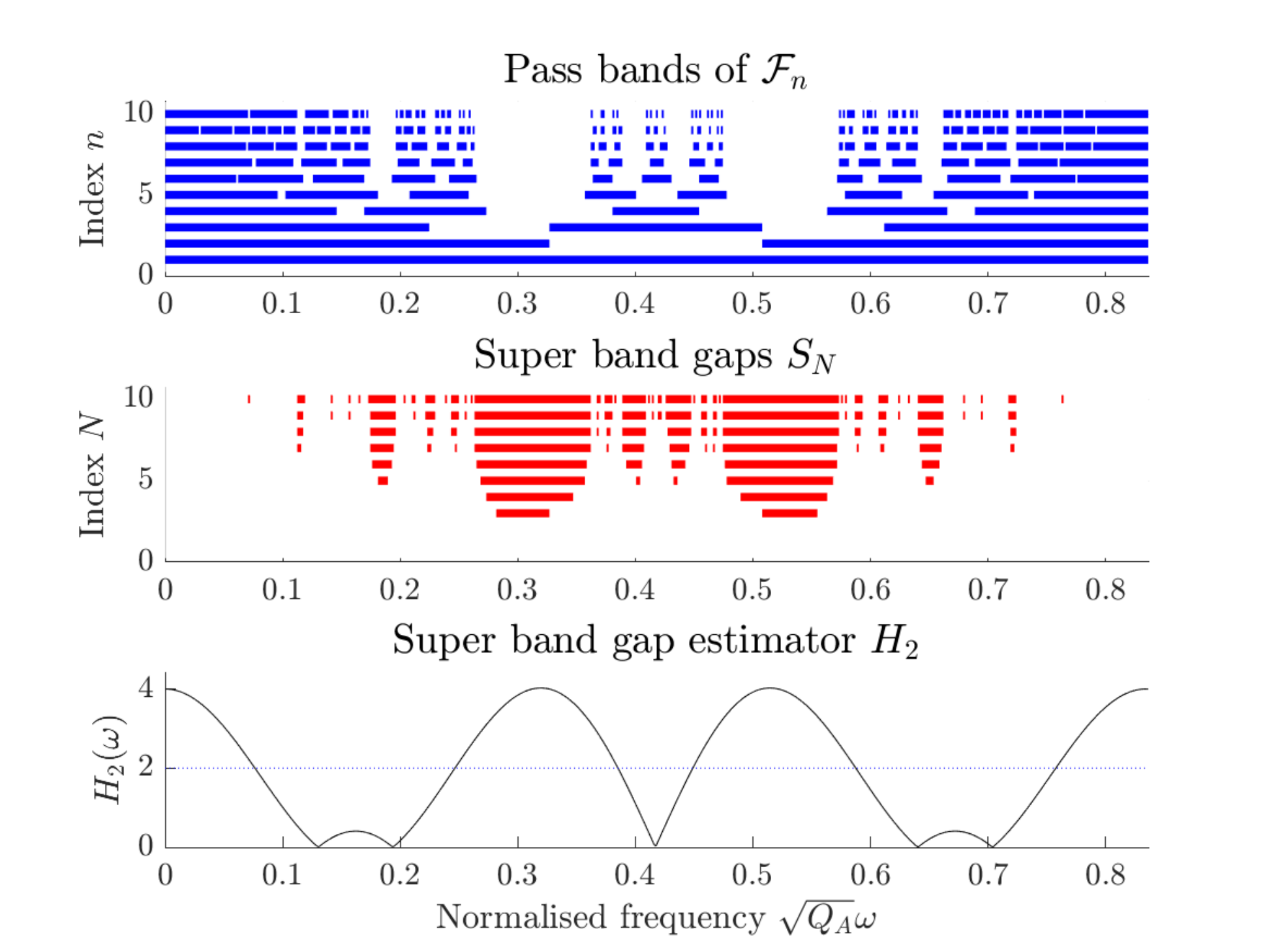}
        \caption{Golden mean Fibonacci ($m=1$, $l=1$).}
    \end{subfigure}
    \hfill
    \begin{subfigure}{0.475\textwidth}
        \includegraphics[width=\linewidth,trim=1cm 0 1.5cm 0,clip]{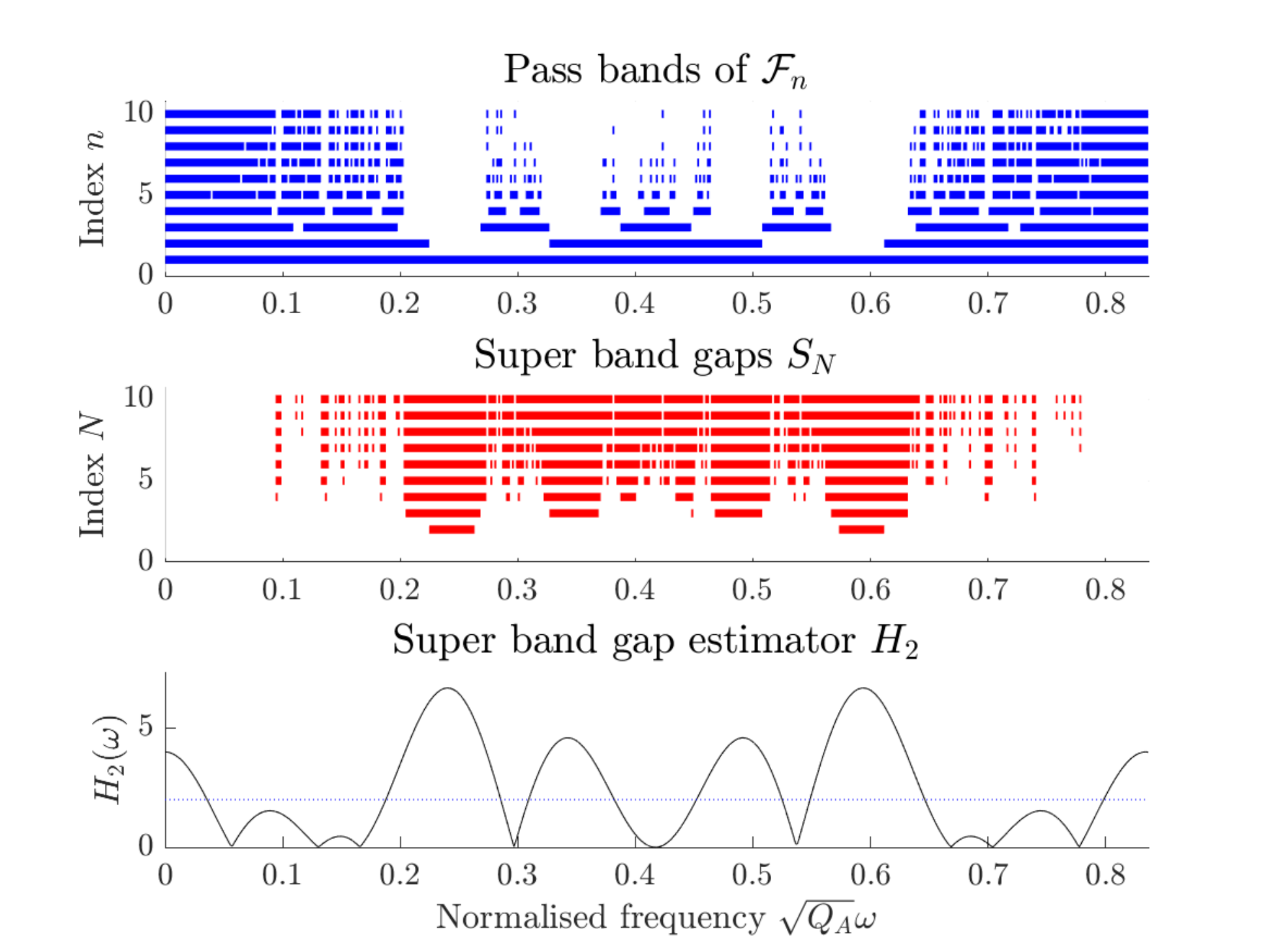}
        \caption{Silver mean Fibonacci ($m=2$, $l=1$).}
    \end{subfigure}

    \vspace{0.2cm}

    \begin{subfigure}{0.475\textwidth}
        \includegraphics[width=\linewidth,trim=1cm 0 1.5cm 0,clip]{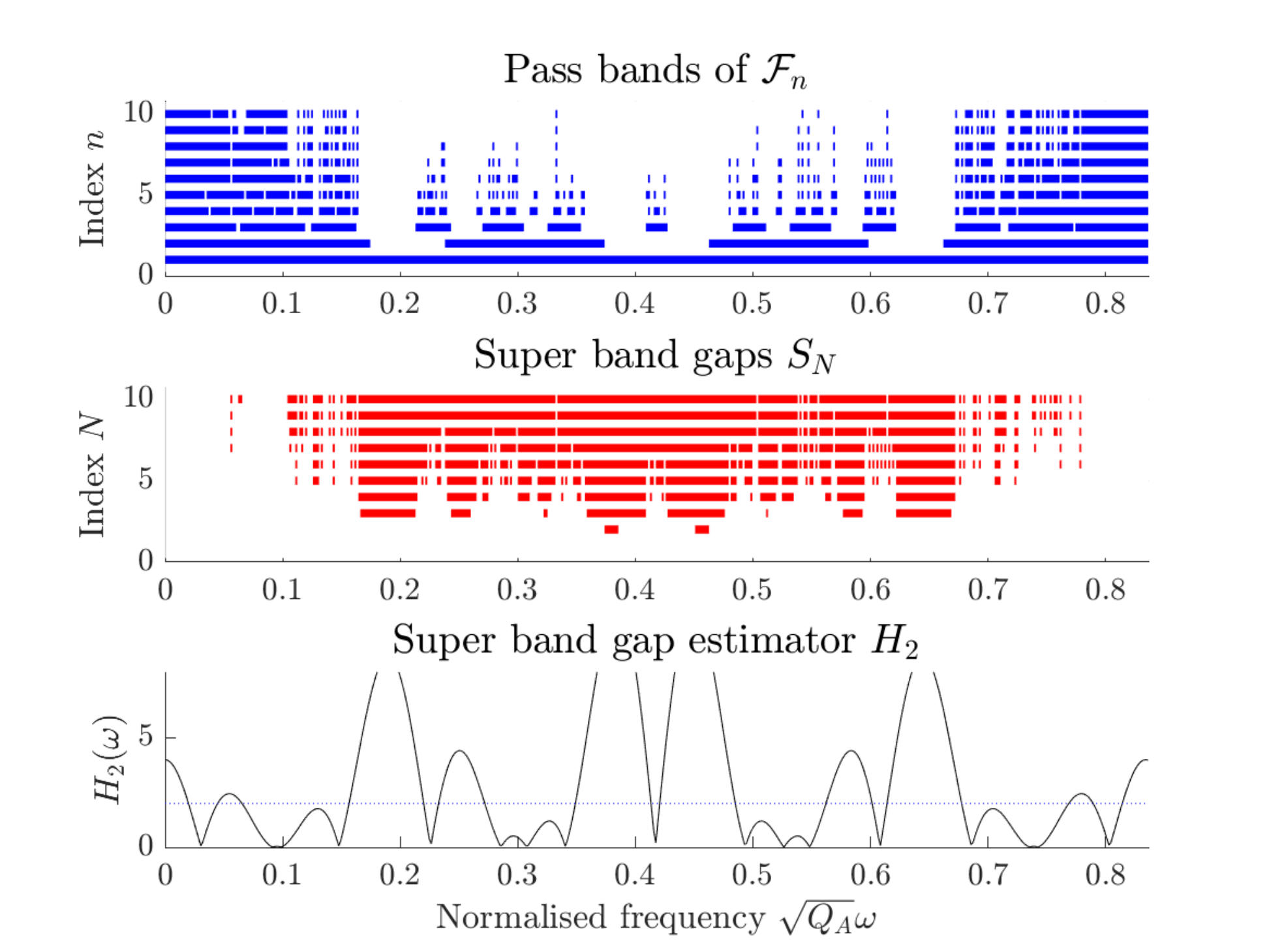}
        \caption{Bronze mean Fibonacci ($m=3$, $l=1$).}
    \end{subfigure}
    \hfill
    \begin{subfigure}{0.475\textwidth}
        \includegraphics[width=\linewidth,trim=1cm 0 1.5cm 0,clip]{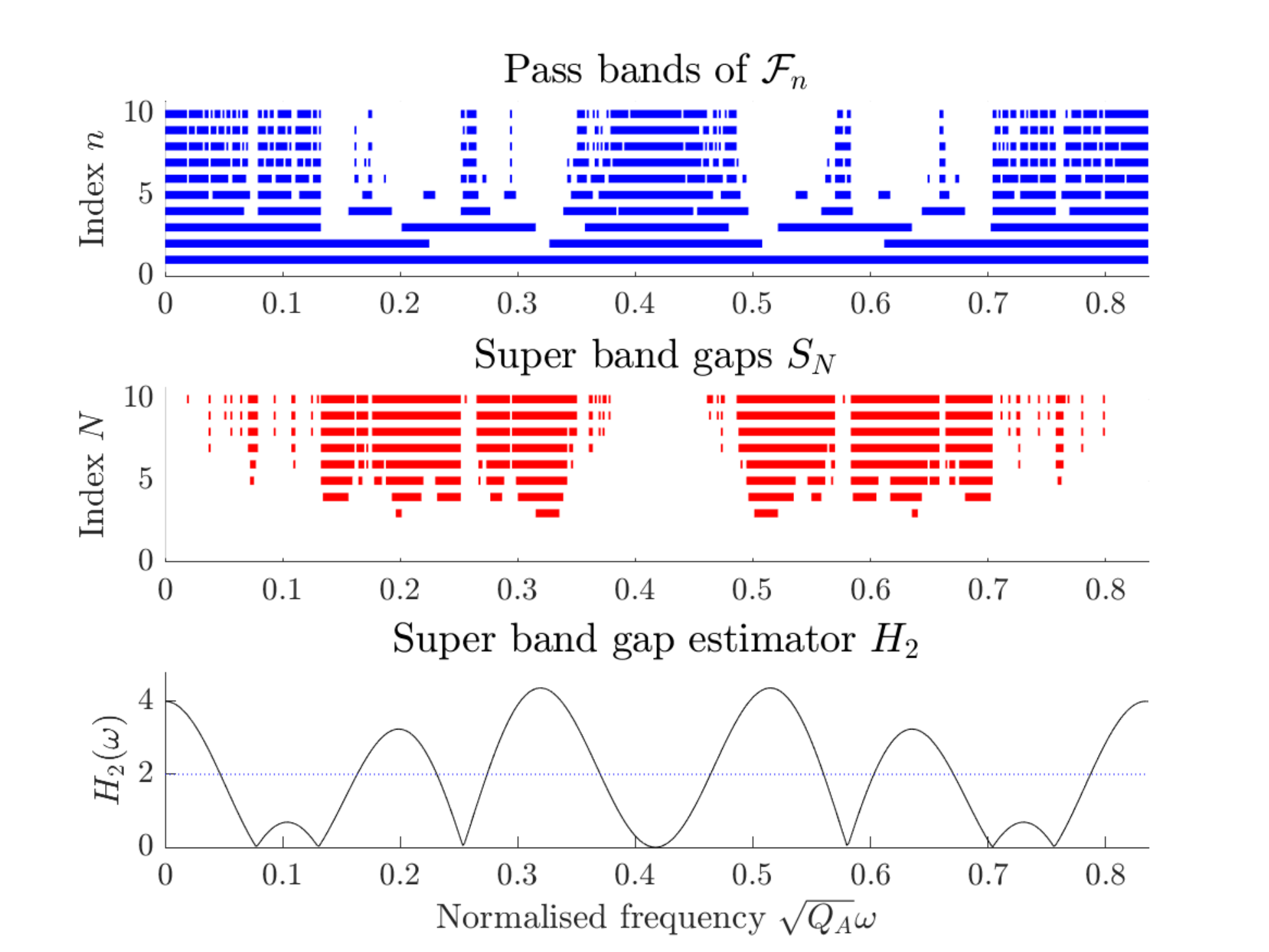}
        \caption{Copper mean Fibonacci ($m=1$, $l=2$).}
    \end{subfigure}
    
    \caption{The pass bands and super band gaps of a system of structured rods with thickness varied according to generalized Fibonacci tilings $\F_n$. For each tiling, the top plot shows the pass bands for each successive Fibonacci tiling $\F_n$, the middle shows the super band gaps $\S_n$, as predicted by the corresponding theorem, and the bottom shows the super band gap estimator $H_2$, as used in previous works and defined in \eqref{eq:Hdefn}. We use the parameter values $E_A=E_B=3.3$GPa, $\rho_A=\rho_B=1140$kg/m$^3$, $2A_A=A_B=1.963\times20^{-3}$m$^2$, $l_A=l_B=0.07$m. We plot the normalised frequency $\sqrt{Q_A}\omega$ on the horizontal axes, noting that $Q_A=Q_B$ in this case.}
    \label{fig:SBG_rod}
\end{figure}

The pattern of pass and stop bands for this continuous system of structured rods is shown in Figure~\ref{fig:SBG_rod} for several generalized Fibonacci tilings. As for the mass-spring system, we show the pattern of pass bands for successive tilings $\F_n$ in the top subplot. Beneath this, we show the frequencies that are guaranteed to lie within super band gaps, thanks to the theorems from Section~\ref{gapstheory}. We see good agreement between the super band gaps $\S_N$ and the gaps between the pass bands of $\F_n$. Once again, we see that as $N$ increases, the super band gaps $\S_N$ recover not only the main band gaps but also a more intricate pattern of super band gaps. 

One notable feature of the spectra in Figure~\ref{fig:SBG_rod} is that they are symmetric and periodic. This is a consequence of the specific setup we have chosen for these simulations, which has all the material parameters identical between $A$ and $B$ (\emph{i.e.} $E_A=E_B$, $\rho_A=\rho_B$ and $l_A=l_B$) and only the cross-sectional area modulated. As a result, the first three terms of the sequence of traces are given by
\begin{equation}
    x_0(\omega) = x_1(\omega) = 2\cos(\sqrt{Q_A}\omega l_A),\quad 
    x_2(\omega) = 2\cos^2(\sqrt{Q_A}\omega l_A) + \bigg( \frac{A_A}{A_B}+\frac{A_B}{A_A} \bigg)\sin^2(\sqrt{Q_A}\omega l_A).
\end{equation}
It is easy to see that these functions are all periodic functions of $\omega$. This spectral symmetry and periodicity was explored through the symmetries of a coordinate transformation in \cite{gei2020phononic}, where they referred to this setup as the ``canonical configuration''. 

\subsection{Flexural waves in continuous beams with modulated supports}
As a third prototype of one-dimensional Fibonacci-generated dynamical systems, we investigate the dispersive properties of flexural vibrations in a quasiperiodic multi-supported beam. In this case, we modulate the distances between the positions of the supports along the axis of the beam (see Figure~\ref{fig:systems}$/(c)$), choosing the lengths according to generalised Fibonacci tilings. The beam is homogeneous, with bending stiffness denoted by $EI$, and the equation governing harmonic vibrations of the transverse displacement $v(z)$ is
\begin{equation}
EIv^{''''}-\rho\omega^2v=0.
\label{beameq}
\end{equation}

The solution of (\ref{beameq}) can be expressed as $v(z)=C\exp{ikz}$, yielding the characteristic equation
\begin{equation} 
(kr)^4-P\omega^2=0,
\label{eqchar}
\end{equation}
where $r$ is the radius of inertia of the cross section and $P=\rho r^4/EI$. Equation
(\ref{eqchar}) admits four solutions, namely
\begin{equation} \label{eq:Kdefn}
k_{1, 2}(\omega) = \pm \frac{1}{r}\sqrt{\omega\sqrt{P}}, \qquad k_{3, 4}(\omega) = \pm \frac{1}{r}\sqrt{-\omega\sqrt{P}},
\end{equation}
where the first index corresponds to the sign $'+'$.

We can now obtain the dispersion diagrams following the same procedure shown in previous subsection for axial waves in structured rods. To do so, it is important to emphasise that the state of the multi-supported beam is determined by the rotation $\phi(z)$ and its derivative $\phi^{'}(z)$ (or bending moment) at each supported point. This is because we assume that the beam is constrained to the support and there is no displacement there. This means that the fourth-order differential system \eqref{eqchar} only has two degrees of freedom. This setting is well established and widely studied, see also \cite{gei2010}. The state vector on the right hand side of the Fibonacci unit cell is then given by $\mathbf{v}_{F_n}=[\phi_{F_n}, \phi^{'}_{F_n}]^{T}$, and it is related to $\mathbf{v}_0=[\phi_0, \phi^{'}_0]^{T}$ through the relationship
\begin{equation}
\mathbf{v}_r=T_n(\omega)\mathbf{v}_l,
\label{multibeamcell}
\end{equation}
where, similarly to the previous cases, $T_n(\omega)=\Pi_{p=1}^{F_{n}}T^X(\omega)$ is the transfer matrix of the unit cell $\mathcal{F}_n$. For this system, the transfer matrices $T^X(\omega)$ ($X\in \left\{A, B\right\}$) associated to each constituent unit are given by \cite{gei2010}
\begin{equation} \label{eq:beammatrix}
T^X(\omega)=
\left[
\begin{array}{cc}
\cfrac{\varPsi_{bb}^{X}(\omega)}{\varPsi_{ab}^{X}(\omega)} & \varPsi_{ba}^{X}(\omega)-\cfrac{\varPsi_{bb}^{X}(\omega)\varPsi_{aa}^{X}(\omega)}{\varPsi_{ab}^{X}(\omega)}\\
\cfrac{1}{\varPsi_{ab}^{X}(\omega)} & -\cfrac{\varPsi_{aa}^{X}(\omega)}{\varPsi_{ab}^{X}(\omega)}
\end{array}
\right],
\end{equation}
where
\begin{eqnarray}
\varPsi_{aa}^{X}(\omega) & = & \frac{k_1(\omega)\cot(k_1(\omega)l_X)-k_3(\omega)\cot(k_3(\omega)l_X)}{k_3^2(\omega)-k_1^2(\omega)},\qquad 
\varPsi_{bb}^{X}(\omega) = -\varPsi_{aa}^{X}(\omega), \label{eq:Psi_aa}\\ 
\varPsi_{ab}^{X}(\omega) & = &  \frac{k_1(\omega)\csc(k_1(\omega)l_X)-k_3(\omega)\csc(k_3(\omega)l_X)}{k_1^2(\omega)-k_3^2(\omega)}, \qquad
\varPsi_{ba}^{X}(\omega) = -\varPsi_{ab}^{X}(\omega),\label{eq:Psi_ab} 
\end{eqnarray}
and $l_X$ ($X\in \left\{A, B\right\}$) is the length of the simply supported beam $A$ or $B$, representing the single element of our cells.

It is important to note that $\varPsi_{aa}^X(\omega)$ and $\varPsi_{ab}^X(\omega)$ both take only real values. This is because, although $k_3(\omega)$ is always an imaginary number, each of $k_3^2$, $k_3\cot(k_3l_X)$ and $k_3\csc(k_3l_X)$ are real. This means $T^X$ always has real-valued entries. Further, we can algebraically check that $T_n(\omega)$ satisfies the unimodularity condition and follows the recursive rule previously introduced. As a consequence, using the Floquet-Bloch condition $\mathbf{v}_r=\mathbf{v}_l e^{iKL_n}$ into equation (\ref{multibeamcell}), we derive a dispersion relation similar to (\ref{disp}).

\begin{figure}
    \centering
    \begin{subfigure}{0.475\textwidth}
        \includegraphics[width=\linewidth,trim=1cm 0 1.5cm 0,clip]{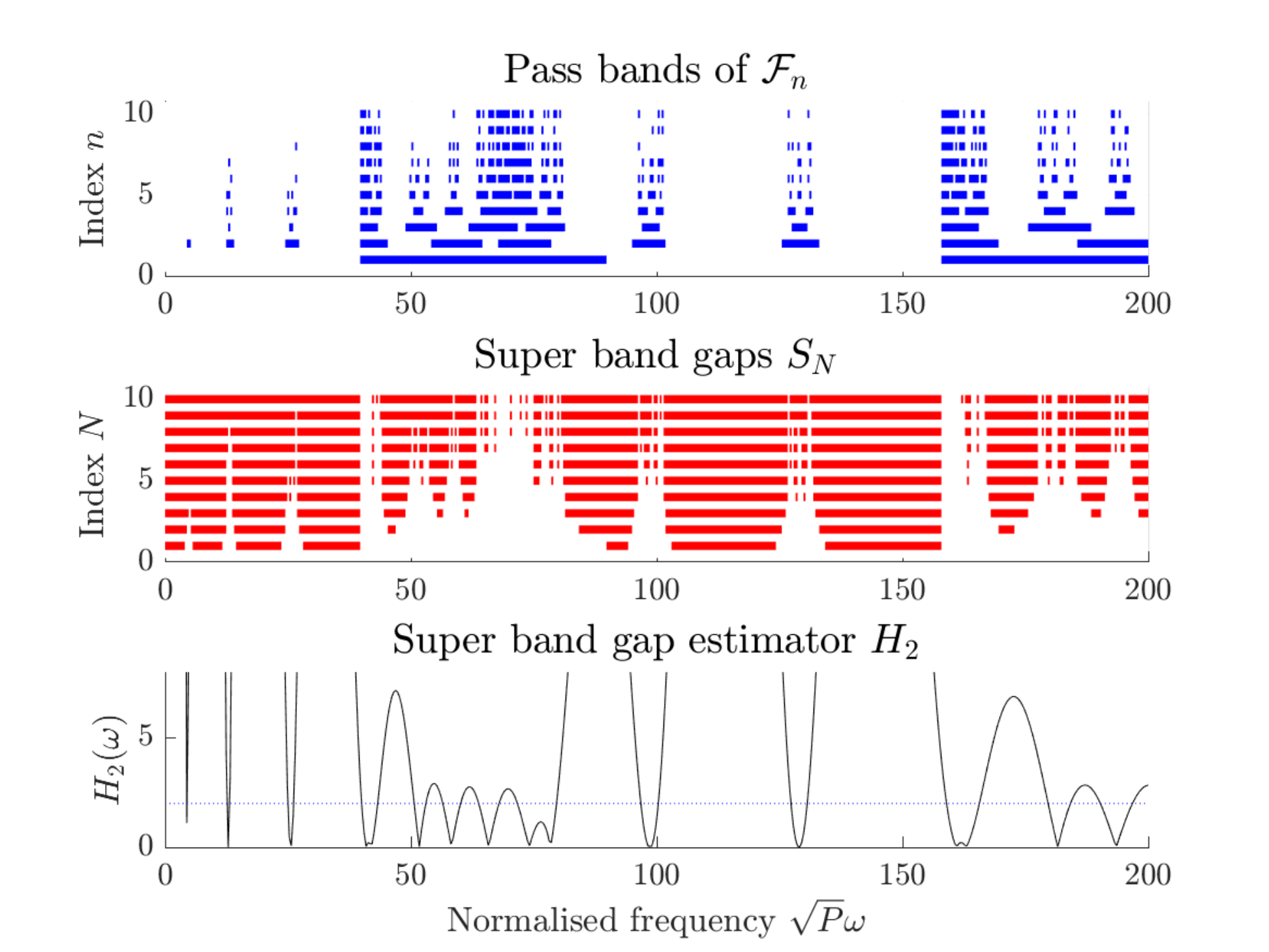}
        \caption{Golden mean Fibonacci ($m=1$, $l=1$).}
    \end{subfigure}
    \hfill
    \begin{subfigure}{0.475\textwidth}
        \includegraphics[width=\linewidth,trim=1cm 0 1.5cm 0,clip]{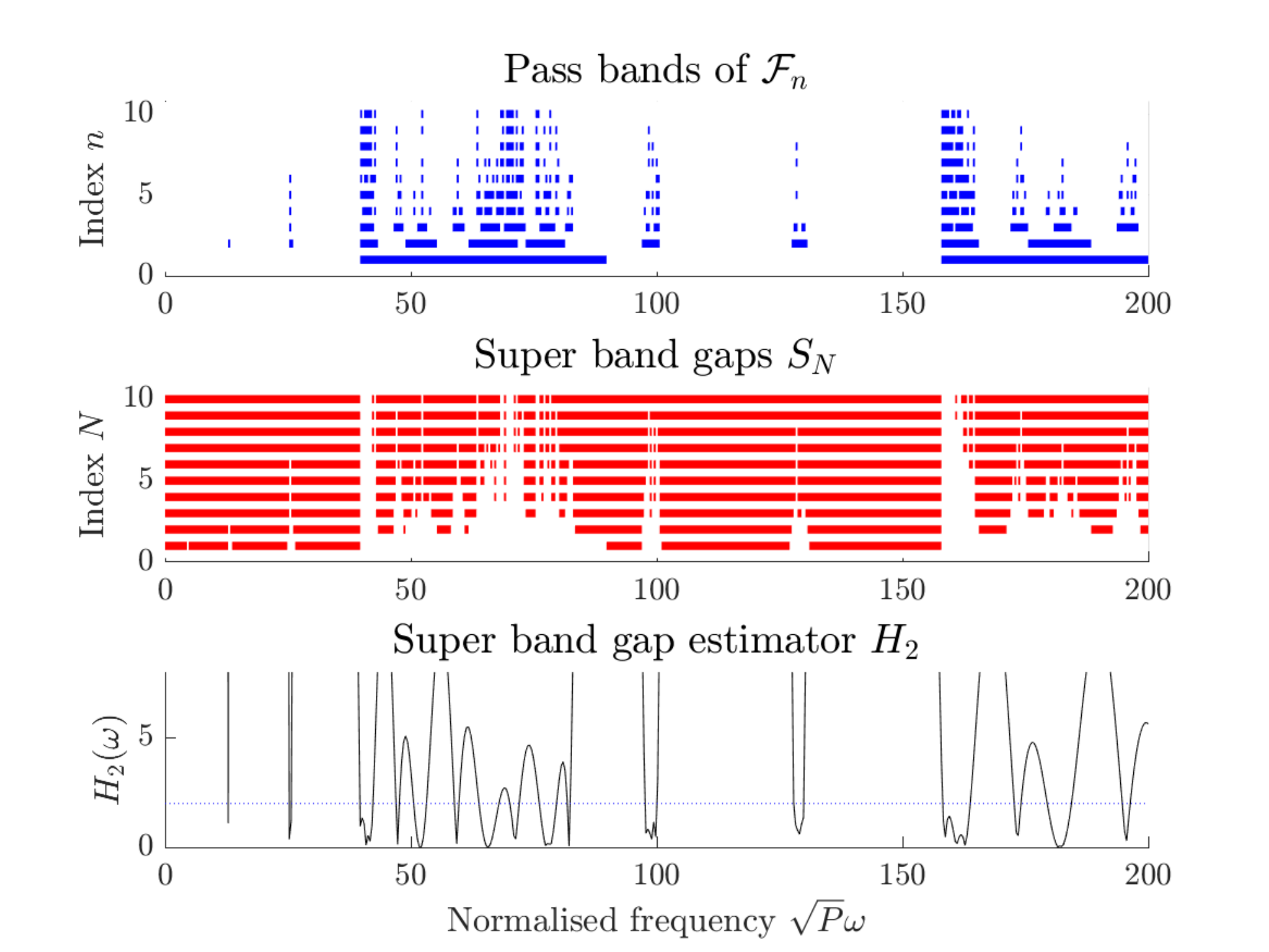}
        \caption{Silver mean Fibonacci ($m=2$, $l=1$).}
    \end{subfigure}

    \vspace{0.2cm}

    \begin{subfigure}{0.475\textwidth}
        \includegraphics[width=\linewidth,trim=1cm 0 1.5cm 0,clip]{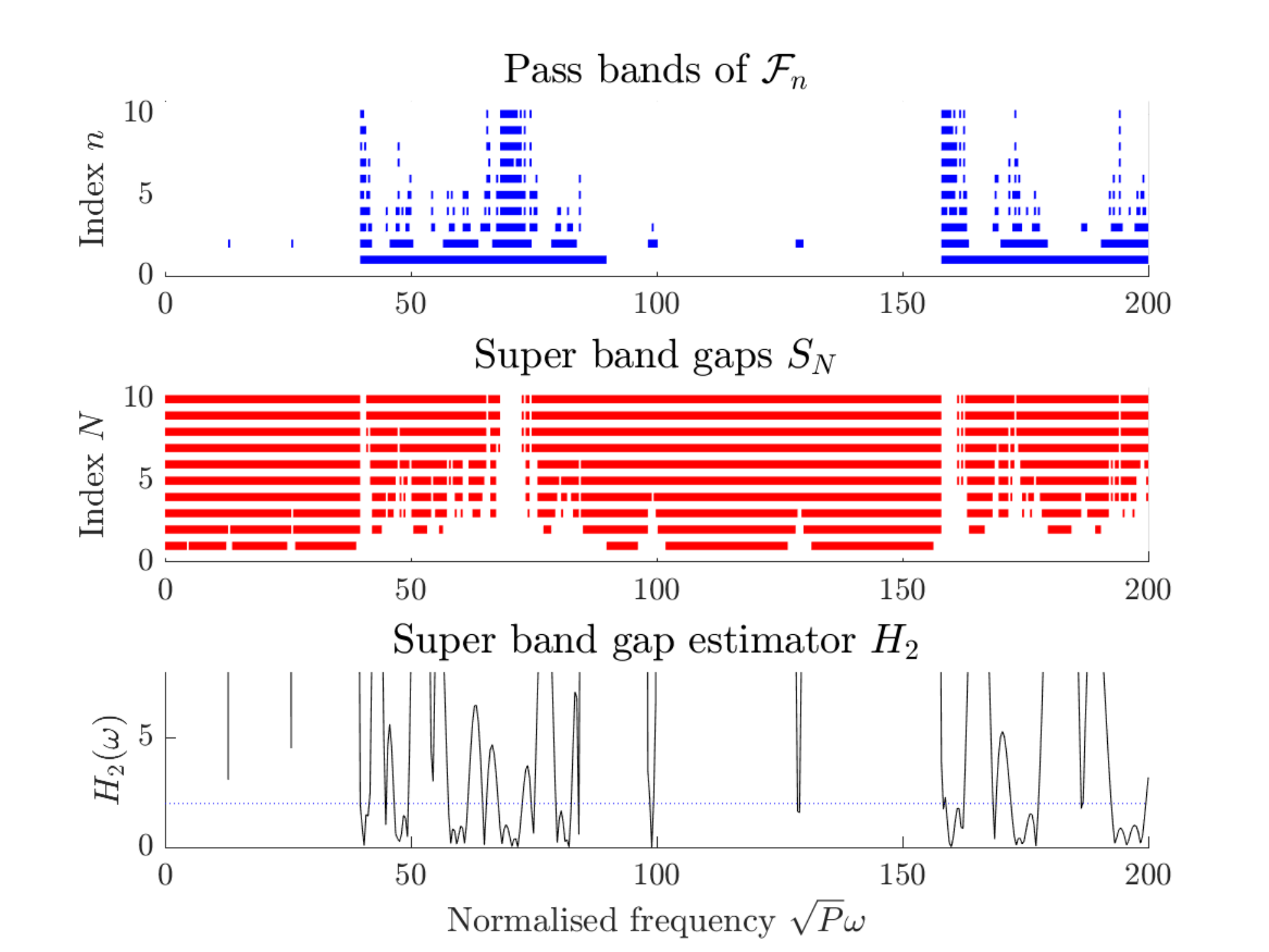}
        \caption{Bronze mean Fibonacci ($m=3$, $l=1$).}
    \end{subfigure}
    \hfill
    \begin{subfigure}{0.475\textwidth}
        \includegraphics[width=\linewidth,trim=1cm 0 1.5cm 0,clip]{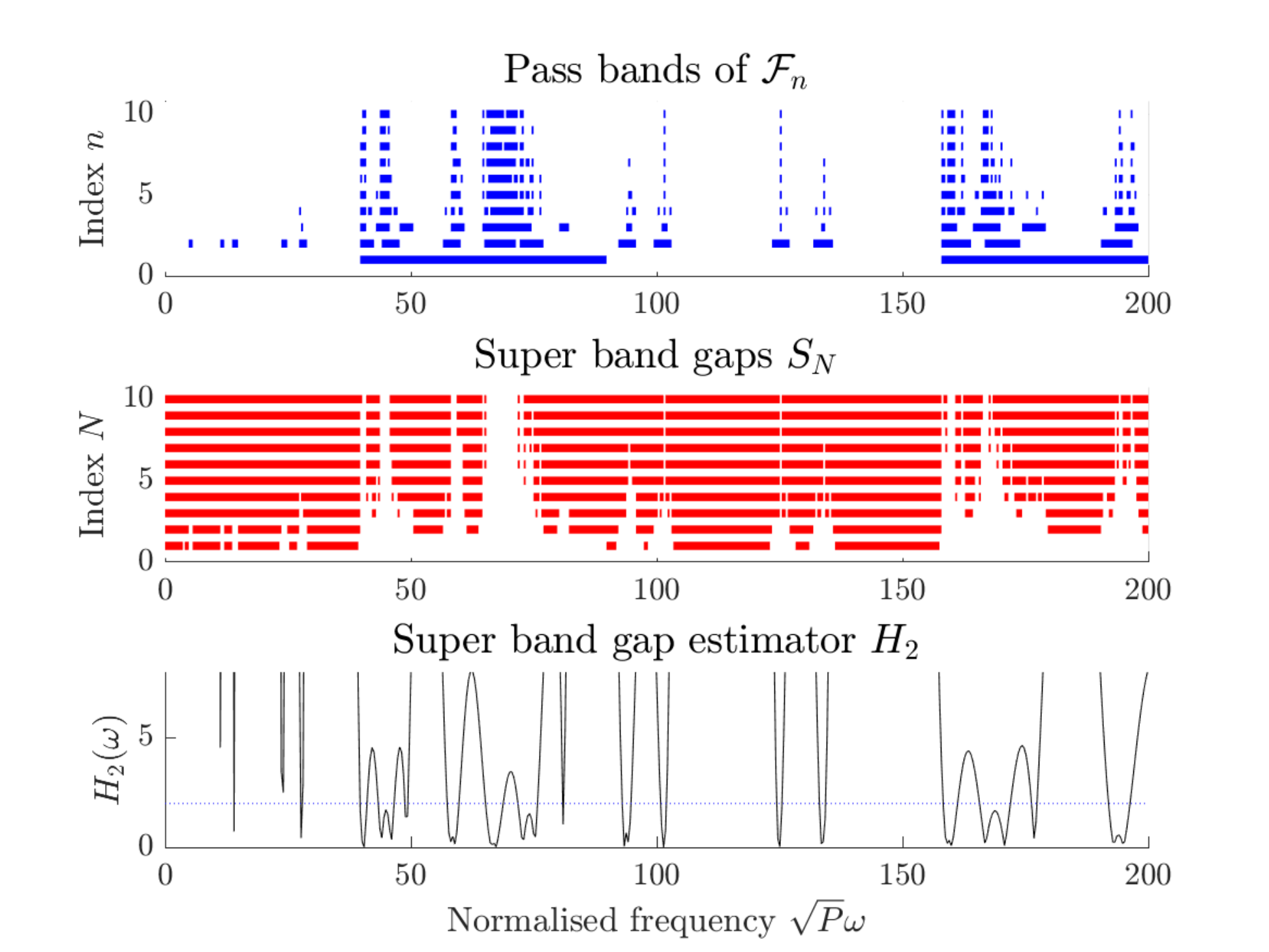}
        \caption{Copper mean Fibonacci ($m=1$, $l=2$).}
    \end{subfigure}
    
    \caption{The pass bands and super band gaps of a continuous beam with supports modulated according to generalized Fibonacci tilings $\F_n$. For each tiling, the top plot shows the pass bands for each successive Fibonacci tiling $\F_n$, the middle shows the super band gaps $\S_n$, as predicted by the corresponding theorem, and the bottom shows the super band gap estimator $H_2$, as used in previous works and defined in \eqref{eq:Hdefn}. We use the distances $4l_A=l_B=0.1$m between the supports and assume that all the material parameters are constant. In particular, we take $r=0.05$m and plot the normalised frequency $\sqrt{P}\omega$ on the horizontal axes.}
    \label{fig:SBG_beam}
\end{figure}

The pattern of pass and stop bands for this continuous system of multi-supported beams is shown in Figure~\ref{fig:SBG_beam} for several generalised Fibonacci tilings. As with the previous examples, a complex pattern of band gaps emerges and the super band gap theory is able to accurately predict this complex behaviour. In this case, the super band gap estimator $H_2$, that was developed in previous works (and is shown in the lower subplots), particularly struggles to reveal useful information about the detailed structure of the spectrum, demonstrating the value of our new theory.

A notable feature of the spectra in Figure~\ref{fig:SBG_beam} is the occurrence of low-frequency super band gaps. As was the case for the high-frequency super band gaps that occurred in the mass-spring system, this can be understood by looking at the structure of the transfer matrices. We recall the function $\sgn:\mathbb{R}\to\{-1,0,1\}$ given by $\sgn(x)=x/|x|$ (and $\sgn(0)=0$). Then, we introduce the sets of unimodular matrices $\Sigma_+$ and $\Sigma_-$ given by
\begin{align}
    \Sigma_+&:=\left\{ M\in\mathbb{R}^{2\times2}: \det(M)=1, \, \sgn(M_{11})=\sgn(M_{22})=1 \text{ and }  \sgn(M_{12})=\sgn(M_{21})=-1 \right\},\\
    \Sigma_-&:=\left\{ M\in\mathbb{R}^{2\times2}: -M\in\Sigma_+ \right\}.
\end{align}

\begin{lemma} \label{lem:sigminus}
    Let $T^X(\omega)$ be the transfer matrix of the multi-supported beam, as defined in \eqref{eq:beammatrix}. There exists some $\omega^{*,X}>0$ such that $T^X(\omega)\in\Sigma_-$ for all $0<\omega<\omega^{*,X}$. Further, it holds that
    \begin{equation*}
        T^X(\omega)=
        \left[\begin{array}{cc}
        -2 & {l_X}/{2}\\ {6}/{l_X} & -2 \end{array}\right] +O(\omega), \quad\text{as }\omega\to0.
    \end{equation*}
\end{lemma}
\begin{proof}
    Suppose that $\omega\to\infty$ while all the other parameters are kept constant. Recalling \eqref{eq:Kdefn}, we see that $k_i=O({\omega}^{1/2})$ and, using the Taylor series for $\cot$ and $\csc$,
    \begin{align}
        k_1(\omega)\cot(k_1(\omega)l_X)=\frac{1}{l_X}-\frac{k_1^2l_X}{3}+O(\omega^2)
        \quad\text{and}\quad
        k_1(\omega)\csc(k_1(\omega)l_X)=\frac{1}{l_X}-\frac{k_1^2l_X}{6}+O(\omega^2),
    \end{align}
    as $\omega\to0$. Substituting these expressions into \eqref{eq:Psi_aa} and \eqref{eq:Psi_ab} gives us that
    \begin{equation}
        \varPsi_{aa}^X = \frac{l_X}{3}+O(\omega)
        \quad\text{and}\quad
        \varPsi_{ab}^X = \frac{l_X}{6}+O(\omega),
    \end{equation}
    as $\omega\to0$. Substituting this into the expression \eqref{eq:beammatrix} we obtain the leading-order expression for $T^X$. Since the leading-order matrix is in $\Sigma_-$, $T^X$ will be in $\Sigma_-$ provided $\omega$ is sufficiently small.
\end{proof}

\begin{lemma} \label{lem:sigmaparity}
    Suppose that $0<\omega<\min\{\omega^{*,A},\omega^{*,B}\}$ and let $T_n$ be the transfer matrix associated to a multi-supported beam with fundamental cell designed according to a generalised Fibonacci substitution rule \eqref{eq:tiling} with arbitrary $m,l\geq1$. $T_n\in\Sigma_-$ if $F_n$ is odd and $T_n\in\Sigma_+$ if $F_n$ is even.
\end{lemma}
\begin{proof}
    From Lemma~\ref{lem:sigminus}, we have that both $T^A\in\Sigma_-$ and $T^B\in\Sigma_-$. It is straightforward to verify that
    \begin{equation}
        \Sigma_-\otimes\Sigma_-= \Sigma_+\otimes\Sigma_+=\Sigma_+ \quad\text{and} \quad 
        \Sigma_-\otimes\Sigma_+= \Sigma_+\otimes\Sigma_-=\Sigma_-.
    \end{equation}
    Then, if $F_n$ is even, $T_n$ is the product of an even number of matrices from $\Sigma_-$, meaning it is the product of $F_n/2$ matrices from $\Sigma_+$, so $T_n\in\Sigma_-$. Conversely, if $F_n$ is odd, then $T_n$ may be written as the product of $F_n-1$ matrices in $\Sigma_-$ and another matrix in $\Sigma_-$. Since $F_n-1$ is even, the first of these two terms is in $\Sigma_+$, meaning $T_n\in\Sigma_+\otimes\Sigma_-=\Sigma_-$.
\end{proof}

We are now in a position to prove an analogous result to Theorem~\ref{thm:highfreqgaps}, which demonstrates the existence of low-fequency super band gaps for the multi-supported beam. From Lemma~\ref{lem:sigminus}, we can see that $\omega$ will be in a band gap of both $\F_0$ and $\F_1$ if it is sufficiently small. However, as was the case for the discrete system, we must take advantage of the specific structure of the transfer matrices in this regime to prove a result.

\begin{thm}
    Consider a multi-supported beam with behaviour governed by the equation \eqref{multibeamcell} and fundamental cells designed according to a generalised Fibonacci substitution rule \eqref{eq:tiling} with arbitrary $m,l\geq1$. There exists some $\omega^*>0$ such that if $0<\omega<\omega^*$ then $\omega$ is in the super band gap $\S_0$.
\end{thm}
\begin{proof}
    The key to our argument is proving that 
    \begin{equation} \label{eq:diaggrowth}
        |(T_n)_{11}|\geq 2^{F_n} \quad\text{and}\quad |(T_n)_{22}|\geq 2^{F_n}.
    \end{equation}
    We first consider the golden mean Fibonacci case, where $m=l=1$, and proceed by induction. From Lemma~\ref{lem:sigminus}, we can see that \eqref{eq:diaggrowth} holds for both $T_0=T^B$ and $T_1=T^A$. Then, for an arbitrary $n\geq1$, it holds for the golden mean Fibonacci tiling that
    \begin{equation} \label{eq:matmul}
        (T_{n+1})_{11}=(T_{n-1})_{11}(T_{n})_{11}+(T_{n-1})_{12}(T_{n})_{21}
        \quad\text{and}\quad
        (T_{n+1})_{22}=(T_{n-1})_{22}(T_{n})_{22}+(T_{n-1})_{21}(T_{n})_{12}.
    \end{equation}
    Thanks to Lemma~\ref{lem:sigmaparity} we know that $T_{n-1},T_n\in\Sigma_+\cup\Sigma_-$, hence it holds that $\sgn((T_{n-1})_{11}(T_{n})_{11})=\sgn((T_{n-1})_{12}(T_{n})_{21})$ and similarly $\sgn((T_{n-1})_{22}(T_{n})_{22})=\sgn((T_{n-1})_{21}(T_{n})_{12})$. As a result, \eqref{eq:matmul} gives us that 
    \begin{align} 
        |(T_{n+1})_{11}|&>|(T_{n-1})_{11}(T_{n})_{11}|\geq 2^{F_{n-1}+F_n}=2^{F{n+1}},\\
        |(T_{n+1})_{22}|&>|(T_{n-1})_{22}(T_{n})_{22}|\geq 2^{F_{n-1}+F_n}=2^{F{n+1}}.
    \end{align}
    Then, we can proceed by induction to conclude that \eqref{eq:diaggrowth} holds for all $n$, for the golden mean Fibonacci case. For arbitrary $m,l\geq 1$, we can use a similar argument, where the key step is to realise that the terms in the equivalent expansion to \eqref{eq:matmul} all have the same sign. As a result, we have the desired bounds
    \begin{align} 
        |(T_{n+1})_{11}|&>|(T_{n-1})_{11}^l(T_{n})_{11}^m|\geq 2^{lF_{n-1}+mF_n}=2^{F{n+1}},\\
        |(T_{n+1})_{22}|&>|(T_{n-1})_{22}^l(T_{n})_{22}^m|\geq 2^{lF_{n-1}+mF_n}=2^{F{n+1}},
    \end{align}
    meaning \eqref{eq:diaggrowth} holds for any generalised Fibonacci tiling.
    
    Finally, thanks to Lemma~\ref{lem:sigmaparity}, we know that $(T_n)_{11}$ and $(T_n)_{22}$ must have the same sign. Hence, it follows from \eqref{eq:diaggrowth} that $|\tr(T_n)|\geq 2^{F_n+1}>2$, so $\omega$ must be in a band gap for all $n$.
\end{proof}

\section{Periodic approximants} \label{sec:periodic}

The aim of this final section is to demonstrate that our theory of super band gaps is not only useful for predicting band gaps in Fibonacci-generated periodic materials, but also for predicting the dynamical properties of real non-periodic quasicrystalline structures. To this end, we take a finite-sized piece of a one-dimensional Fibonacci quasicrystal and compare its transmission coefficient with the stop/pass band diagrams obtained by applying the Floquet-Bloch theory to infinite periodic waveguides generated according to consecutive Fibonacci cells $\mathcal{F}_n$. We will present results for the case of a structured rod, as studied in Section~\ref{sec:rods}, but it is reasonable to expect similar behaviour for the other physical systems also.

\begin{figure}[htbp]
    \centering
    \includegraphics[width=0.95\linewidth]{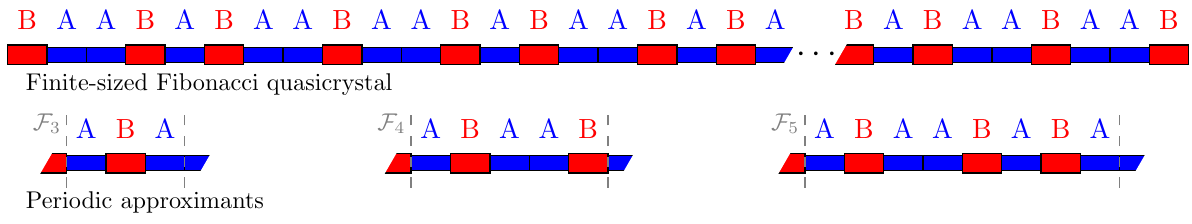}
    \caption{Our results show that the main spectral gaps of a Fibonacci quasicrystal can be faithfully predicted by periodic approximants. We compare the transmission coefficient of a quasiperiodic structured rod of finite length with the Bloch spectra of periodic approximants.}
    \label{fig:finitesystems}
\end{figure}

The Fibonacci quasicrystal we take, as a demonstrative example, is a finite rod formed by joining together golden mean cells $\F_0$, $\F_1$ all the way up to $\F_6$. This gives a structure composed of 32 different phases $A$ and $B$, as depicted in Figure~\ref{fig:finitesystems}. Considering axial vibrations propagating in this system, the global transfer matrix is defined as $T_{G}(\omega)=\Pi_{n=1}^{6}T_n(\omega)$, where $T_n(\omega)$ are the matrices associated with the cells $\F_n$ that were introduced in Section~\ref{sec:rods}. According to the method adopted in \cite{gei2020phononic}, it can be shown that the trasmission coefficient for a finite quasicrystalline sample is given by 
\begin{equation}
    T_c(\omega)=\frac{u_l}{u_r}=\frac{1}{T_{G22}(\omega)},
\end{equation}
where $T_{G22}$ is the lower-right entry of the $2\times2$ square matrix $T_{G}$.  

In Figure~\ref{fig:compFB.pdf}, the transmission coefficient $T_c(\omega)$ for the finite quasicrystalline rod is plotted using a logarithmic scale. In each of the four plots, this is compared with the super band gaps predicted by $\mathcal{F}_2, \mathcal{F}_3, \mathcal{F}_4$ and $\mathcal{F}_5$ (\emph{i.e.} the sets $\S_2$, $\S_3$, $\S_4$ and $\S_5$, to use the notation from Section~\ref{gapstheory}). For these numerical computations we adopted a setup that leads to a periodic and symmetric spectrum, as mentioned in  Section~\ref{sec:rods} and referred to as the ``canonical configuration'' in \cite{gei2020phononic}. Therefore, the results reported for one period describe the dispersion properties for the whole range of real frequencies. We observe that, as the order of the Fibonacci unit cells increases, the super band gaps given by the periodic rods (denoted by the grey shaded areas) closely match the frequency intervals where the transmission coefficient is small, corresponding to a significant attenuation of the propagation in the finite structure, until they become almost coincident for $\mathcal{F}_5$. This demonstrates that the super band gaps corresponding to a periodic infinite rod with a relatively short fundamental cell approximate with excellent accuracy the spectrum of finite non-periodic quasicrystalline structures.

\begin{figure}[htbp]
    \centering
    \includegraphics[scale=0.8]{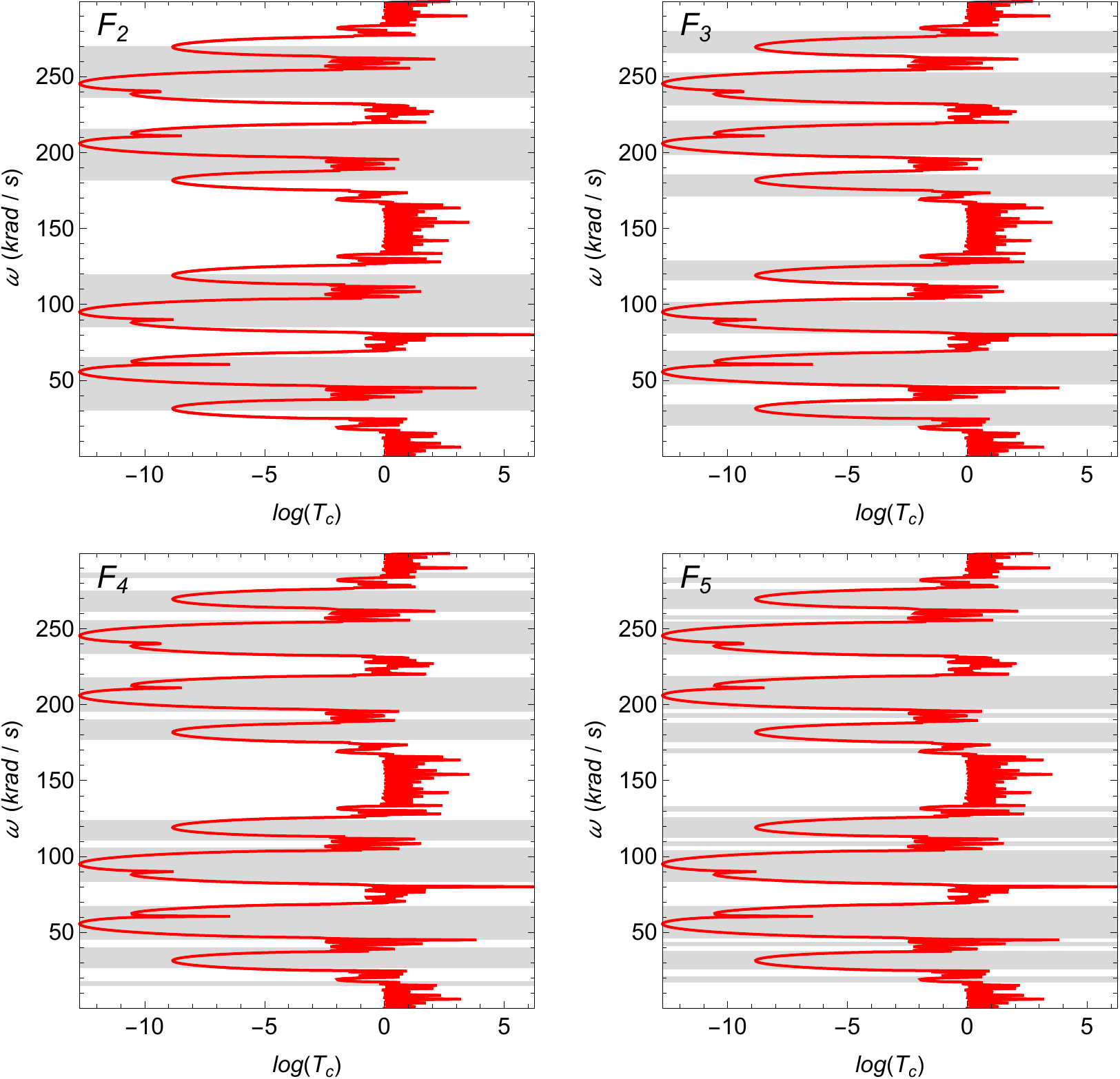}
    \caption{Transmission coefficient for a finite quasicrystalline rod composed of golden mean cells $\mathcal{F}_0$ to $\mathcal{F}_6$ (red line) compared with the super band gaps of infinite structures generated according to $\mathcal{F}_2$, $\mathcal{F}_3$, $\mathcal{F}_4$ and $\mathcal{F}_5$ (grey shaded areas). We use the parameter values $E_A=E_B=3.3$GPa, $\rho_A=\rho_B=1140$kg/m$^3$, $4A_A=A_B=1.963\times20^{-3}$m$^2$, $l_A=2l_B=0.07$m. The frequency range is shown on the vertical axis and the logarithm of the transmission coefficient $\log{T_c}$ is plotted on the horizontal axis.}
    \label{fig:compFB.pdf}
\end{figure}

In Figure~\ref{fig:compSpectrum.pdf}, the same transmission coefficient $T_c(\omega)$ is plotted, but this time we compare it to the transmission spectrum of a periodic approximant. In this case, the transmission coefficient for a finite-sized piece of periodic material is shown with a dotted line, again for the simple golden mean Fibonacci tilings. The finite pieces of periodic material are composed of $\mathcal{N}$ elementary cells $\mathcal{F}_2, \mathcal{F}_3, \mathcal{F}_4$ and $\mathcal{F}_5$. The global transfer matrix for these finite periodic rods is then defined as $T_{G}(\omega)=T_n^{\mathcal{N}}(\omega)$. The results reported in Figure~\ref{fig:compSpectrum.pdf} are obtained assuming $\mathcal{N}=7$, and the number of elements composing the samples is $\tilde{F}_n=7F_n$, where $F_n$ is the Fibonacci golden number corresponding to the phases contained in $\mathcal{F}_n$ (i.e. for $\mathcal{F}_2$ finite rod, $F_2=2$ and $\tilde{F}_2=14$). Even when the periodic approximant has a small unit cell (so the approximation is relatively crude), such as for example in the case of $\mathcal{F}_3$ ($F_3=3$ and $\tilde{F}_3=21$), the main spectral gaps are accurately predicted. This is naturally explained by our theory for super band gaps, which demonstrates the existence of frequency ranges which will always be in spectral gaps, for any size of Fibonacci quasicrystal.

\begin{figure}[htbp]
    \centering
    \includegraphics[scale=0.8]{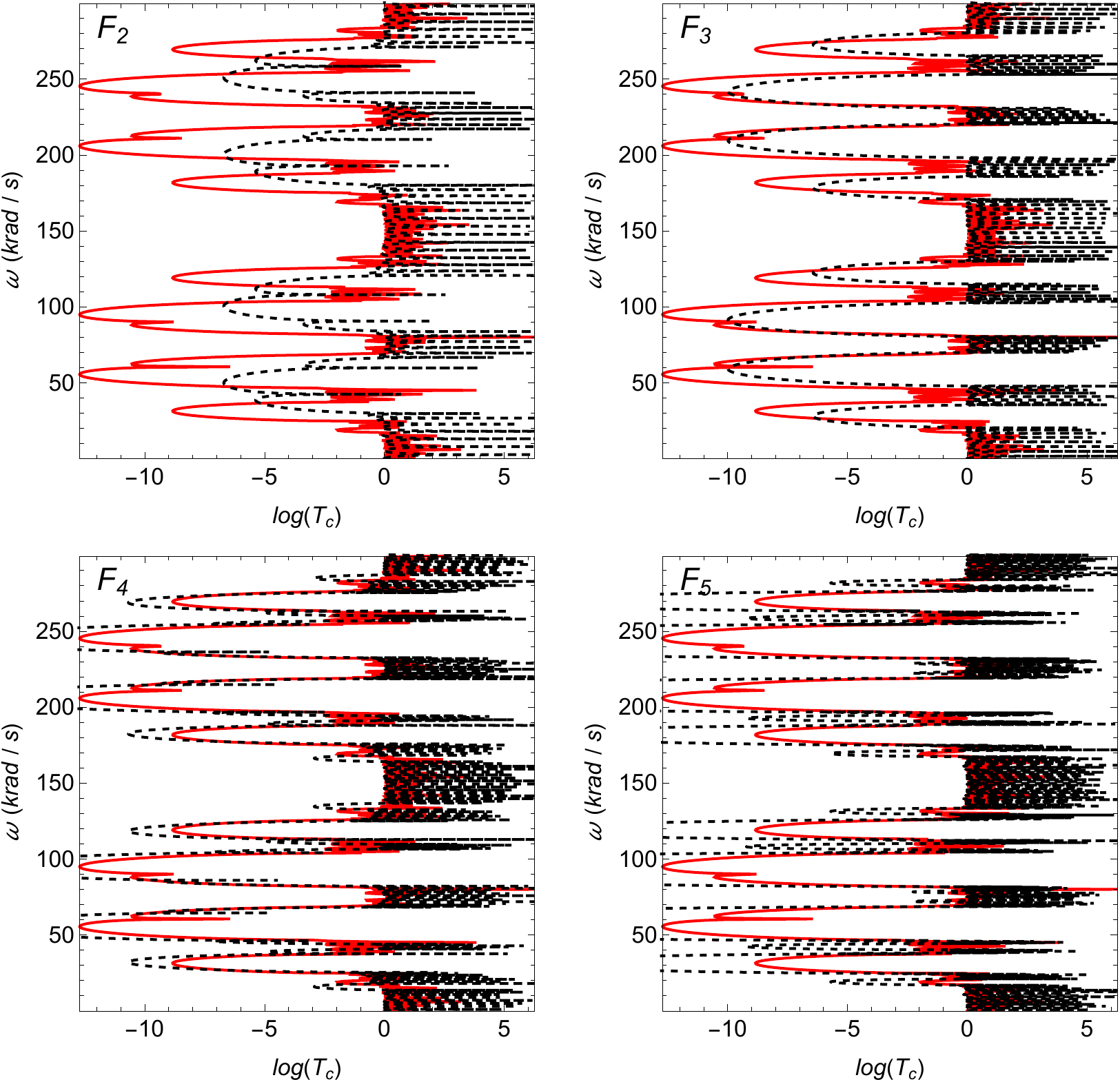}
    \caption{Transmission coefficients for a finite quasicrystalline rod composed of golden mean cells $\mathcal{F}_0$ to $\mathcal{F}_6$ (red line) and for finite periodic samples consisting in seven $\mathcal{F}_2$, $\mathcal{F}_3$, $\mathcal{F}_4$ and $\mathcal{F}_5$ cells (black dashed lines). We use the parameter values $E_A=E_B=3.3$GPa, $\rho_A=\rho_B=1140$kg/m$^3$, $4A_A=A_B=1.963\times20^{-3}$m$^2$, $l_A=2l_B=0.07$m.  The frequency range is shown on the vertical axis and the logarithm of the transmission coefficient $\log{T_c}$ is plotted on the horizontal axis.}
    \label{fig:compSpectrum.pdf}
\end{figure}

\section{Concluding remarks}

We have developed novel mathematical theory for characterising super band gaps in periodic structures generated by generalised Fibonacci tilings. This demonstrates the mechanism through which structural self similarity creates identifiable features in the otherwise complex spectra of quasiperiodic systems. Our results also justify the use of periodic approximants for generalised Fibonacci tilings, as we have proved that the properties of a given element in this sequence of tilings can be used to predict spectral characteristics (band gaps) of all subsequent elements in the sequence. We have demonstrated this by considering a large quasicrystalline material, which is made from several different Fibonacci tiles, and shown that the frequency ranges where its transmission coefficient drops are in close correspondence with the super band gaps predicted by periodic approximants (even with relatively small unit cells).

This work is significant since it provides a concise and computationally efficient way to predict the main spectral features of quasicrystalline materials. This is crucial if such materials are to be used in wave control applications, which has been the subject of several recent studies. For example, Fibonacci tilings have been used as the basis for designing symmetry-induced waveguides \cite{davies2022symmetry} and laminar materials which exhibit negative refraction \cite{morini2019negative}. Similar studies have also been conducted for other quasicrystals, such as variants of the Harper model \cite{apigo2018topological, marti2021edge, pal2019topological}. Understanding a material's spectral gaps is essential to be able to design any such device, and the results in this work (which could be generalised to other quasicrystalline materials generated by tiling rules \cite{grunbaum1987tilings, macia2014nature}) provide a first step for doing so.

\section*{Acknowledgements}

The work of BD was supported by a fellowship from the Engineering and Physical Sciences Research Council with grant number EP/X027422/1. LM thanks the support of Cardiff University.

\section*{Competing interests}

The authors have no competing interests to declare.

\section*{Data availability}

The software developed for this study is available at \url{https://doi.org/10.5281/zenodo.7602934}. No specific datasets were generated or analysed during the current study.

\bibliographystyle{abbrv}
\bibliography{references}{}

\begin{thebibliography}{10}

\bibitem{amenoagbadji2023wave}
P.~Amenoagbadji, S.~Fliss, and P.~Joly.
\newblock Wave propagation in one-dimensional quasiperiodic media.
\newblock {\em arXiv preprint arXiv:2301.01159}, 2023.

\bibitem{apigo2018topological}
D.~J. Apigo, K.~Qian, C.~Prodan, and E.~Prodan.
\newblock Topological edge modes by smart patterning.
\newblock {\em Physical Review Materials}, 2(12):124203, 2018.

\bibitem{avila2009ten}
A.~Avila and S.~Jitomirskaya.
\newblock The ten martini problem.
\newblock {\em Annals of Mathematics}, 170(1):303--342, 2009.

\bibitem{bender1999advanced}
C.~M. Bender and S.~A. Orszag.
\newblock {\em Advanced Mathematical Methods for Scientists and Engineers I:
  Asymptotic Methods and Perturbation Theory}.
\newblock Springer Science \& Business Media, 1999.

\bibitem{chan1998photonic}
Y.~Chan, C.~T. Chan, and Z.~Liu.
\newblock Photonic band gaps in two dimensional photonic quasicrystals.
\newblock {\em Physical Review Letters}, 80(5):956, 1998.

\bibitem{dal2007spectral}
L.~Dal~Negro and N.-N. Feng.
\newblock Spectral gaps and mode localization in {Fibonacci} chains of metal
  nanoparticles.
\newblock {\em Optics Express}, 15(22):14396--14403, 2007.

\bibitem{davies2022symmetry}
B.~Davies and R.~V. Craster.
\newblock Symmetry-induced quasicrystalline waveguides.
\newblock {\em Wave Motion}, 115:103068, 2022.

\bibitem{eliasson1992floquet}
L.~H. Eliasson.
\newblock Floquet solutions for the 1-dimensional quasi-periodic
  {Schr{\"o}dinger} equation.
\newblock {\em Communications in Mathematical Physics}, 146:447--482, 1992.

\bibitem{florescu2009complete}
M.~Florescu, S.~Torquato, and P.~J. Steinhardt.
\newblock Complete band gaps in two-dimensional photonic quasicrystals.
\newblock {\em Physical Review B}, 80(15):155112, 2009.

\bibitem{gei2010}
M.~Gei.
\newblock Wave propagation in quasiperiodic structures: stop/pass band
  distribution and prestress effect.
\newblock {\em International Journal of Solids and Structures}, 517:3067--3075,
  2010.

\bibitem{gei2020phononic}
M.~Gei, Z.~Chen, F.~Bosi, and L.~Morini.
\newblock Phononic canonical quasicrystalline waveguides.
\newblock {\em Applied Physics Letters}, 116(24):241903, 2020.

\bibitem{grunbaum1987tilings}
B.~Gr{\"u}nbaum and G.~C. Shephard.
\newblock {\em Tilings and Patterns}.
\newblock W.H. Freeman, New York, 1987.

\bibitem{gumbs1988dynamical}
G.~Gumbs and M.~K. Ali.
\newblock Dynamical maps, {Cantor} spectra, and localization for {Fibonacci}
  and related quasiperiodic lattices.
\newblock {\em Physical Review Letters}, 60(11):1081, 1988.

\bibitem{hamilton2021effective}
J.~K. Hamilton, M.~Camacho, R.~Boix, I.~R. Hooper, and C.~R. Lawrence.
\newblock Effective-periodicity effects in {Fibonacci} slot arrays.
\newblock {\em Physical Review B}, 104(24):L241412, 2021.

\bibitem{hiramoto1989new}
H.~Hiramoto and M.~Kohmoto.
\newblock New localization in a quasiperiodic system.
\newblock {\em Physical Review Letters}, 62(23):2714, 1989.

\bibitem{jitomirskaya1999metal}
S.~Y. Jitomirskaya.
\newblock Metal-insulator transition for the almost {Mathieu} operator.
\newblock {\em Annals of Mathematics}, 150(3):1159--1175, 1999.

\bibitem{kadic20193d}
M.~Kadic, G.~W. Milton, M.~van Hecke, and M.~Wegener.
\newblock 3d metamaterials.
\newblock {\em Nature Reviews Physics}, 1(3):198--210, 2019.

\bibitem{kohmoto1983localization}
M.~Kohmoto, L.~P. Kadanoff, and C.~Tang.
\newblock Localization problem in one dimension: Mapping and escape.
\newblock {\em Physical Review Letters}, 50(23):1870, 1983.

\bibitem{kohmoto1984cantor}
M.~Kohmoto and Y.~Oono.
\newblock Cantor spectrum for an almost periodic {Schr{\"o}dinger} equation and
  a dynamical map.
\newblock {\em Physics Letters A}, 102(4):145--148, 1984.

\bibitem{kolar1993new}
M.~Kol{\'a}{\v{r}}.
\newblock New class of one-dimensional quasicrystals.
\newblock {\em Physical Review B}, 47(9):5489, 1993.

\bibitem{kolar1989attractors}
M.~Kol{\'a}{\v{r}} and M.~K. Ali.
\newblock Attractors of some volume-nonpreserving {Fibonacci} trace maps.
\newblock {\em Physical Review A}, 39(12):6538, 1989.

\bibitem{kolar1989generalized}
M.~Kol{\'a}{\v{r}} and M.~K. Ali.
\newblock Generalized {Fibonacci} superlattices, dynamical trace maps, and
  magnetic excitations.
\newblock {\em Physical Review B}, 39(1):426, 1989.

\bibitem{kolar1990onedim}
M.~Kol{\'a}{\v{r}} and M.~K. Ali.
\newblock One-dimensional generalized {Fibonacci} tilings.
\newblock {\em Physical Review B}, 41(10):7108, 1990.

\bibitem{kolar1990trace}
M.~Kol{\'a}{\v{r}} and F.~Nori.
\newblock Trace maps of general substitutional sequences.
\newblock {\em Physical Review B}, 42(1):1062, 1990.

\bibitem{kraus2012topological}
Y.~E. Kraus and O.~Zilberberg.
\newblock Topological equivalence between the {Fibonacci} quasicrystal and the
  harper model.
\newblock {\em Physical Review Letters}, 109(11):116404, 2012.

\bibitem{Lazaro2022Sturm}
M.~Lazaro, A.~Niemczynowicz, and L.~M. Garcia-Raffi.
\newblock Elastodynamical properties of {Sturmian} structured media.
\newblock {\em Journal of Sound and Vibration}, 517:116539, 1989.

\bibitem{macia2014nature}
E.~Maci{\'a}.
\newblock On the nature of electronic wave functions in one-dimensional
  self-similar and quasiperiodic systems.
\newblock {\em ISRN Condens. Matter Phys.}, 165943, 2014.

\bibitem{marti2021edge}
M.~Mart{\'\i}-Sabat{\'e} and D.~Torrent.
\newblock Edge modes for flexural waves in quasi-periodic linear arrays of
  scatterers.
\newblock {\em APL Materials}, 9(8):081107, 2021.

\bibitem{morini2019negative}
L.~Morini, Y.~Eyzat, and M.~Gei.
\newblock Negative refraction in quasicrystalline multilayered metamaterials.
\newblock {\em Journal of the Mechanics and Physics of Solids}, 124:282--298,
  2019.

\bibitem{morini2018waves}
L.~Morini and M.~Gei.
\newblock Waves in one-dimensional quasicrystalline structures: dynamical trace
  mapping, scaling and self-similarity of the spectrum.
\newblock {\em Journal of the Mechanics and Physics of Solids}, 119:83--103,
  2018.

\bibitem{moustaj2022spectral}
A.~Moustaj, M.~R{\"o}ntgen, C.~V. Morfonios, P.~Schmelcher, and C.~M. Smith.
\newblock Spectral properties of two coupled {Fibonacci} chains.
\newblock {\em arXiv preprint arXiv:2208.05178}, 2022.

\bibitem{pal2019topological}
R.~K. Pal, M.~I.~N. Rosa, and M.~Ruzzene.
\newblock Topological bands and localized vibration modes in quasiperiodic
  beams.
\newblock {\em New Journal of Physics}, 21(9):093017, 2019.

\bibitem{rodriguez2008computation}
A.~W. Rodriguez, A.~P. McCauley, Y.~Avniel, and S.~G. Johnson.
\newblock Computation and visualization of photonic quasicrystal spectra via
  {Bloch’s} theorem.
\newblock {\em Physical Review B}, 77(10):104201, 2008.

\bibitem{RuiWang2019}
X.~Rui, G.~Wang, and J.~Zhang.
\newblock {\em Transfer Matrix Method for Multibody Systems: Theory and
  Applications}.
\newblock John Wiley \& Sons, Singapore, 2019.

\bibitem{surace1990schrodinger}
S.~Surace.
\newblock The {Schr{\"o}dinger} equation with a quasi-periodic potential.
\newblock {\em Transactions of the American Mathematical Society},
  320(1):321--370, 1990.

\end{thebibliography}
\end{document}